\documentclass[12pt,draft]{article}

\usepackage{amsmath}
\usepackage{amssymb}
\usepackage{amsthm}
\usepackage{geometry}
\usepackage{titlesec}
\usepackage[numbers,sort&compress]{natbib}
\usepackage{tabularx}
\usepackage{booktabs}
\usepackage{mathtools}
\usepackage{multirow}
\usepackage{multicol}
\usepackage{threeparttable}
\usepackage{enumitem}

\numberwithin{equation}{section}
\newtheorem{theorem}{Theorem}[section]
\newtheorem{lemma}[theorem]{Lemma}
\newtheorem{proposition}[theorem]{Proposition}
\newtheorem{corollary}[theorem]{Corollary}
\theoremstyle{definition}
\newtheorem{definition}[theorem]{Definition}
\newtheorem{remark}[theorem]{Remark}
\theoremstyle{plain}
\newtheorem{example}[theorem]{Example}

\titleformat*{\section}{\large\bfseries}
\titleformat*{\subsection}{\normalsize\bfseries}

\newcommand{\bC}{\ensuremath{\mathbb{C}}}
\newcommand{\bCx}{\ensuremath{\mathbb{C}^*}}
\newcommand{\bZ}{\ensuremath{\mathbb{Z}}}
\newcommand{\subno}[1]{\noindent\textnormal{(#1)}}
\newcommand{\afCase}[1]{\textbf{Case~#1:}}

\newcommand{\End}{\mathrm{End }}

\newcommand{\Span}{\mathrm{span }}
\newcommand{\Rad}{\mathrm{Rad }}
\newcommand{\rank}{\mathrm{rank }}
\newcommand{\Lie}{\mathrm{Lie }}
\newcommand{\bZz}{\ensuremath{\mathbb{Z}_2}}
\newcommand{\ep}{\ensuremath{{\bar{0}}}}
\newcommand{\op}{\ensuremath{{\bar{1}}}}
\newcommand{\cp}{\ensuremath{\mathbb{C}[\partial]}}
\newcommand{\Vir}{\mathrm{Vir }}
\newcommand{\CH}{\mathcal{H }}
\newcommand{\Aut}{\mathrm{Aut }}

\begin{document}
	\begin{center}
		\bfseries\large
		Lie conformal superalgebras of rank $(2+1)$
		\footnote{Supported by the National Natural Science Foundation of China (No. 12471027) and Natural Science Foundation of Shanghai (No. 24ZR1471900).
			\\ \indent \ \ $^*$Corresponding author: Xiaoqing Yue (xiaoqingyue@tongji.edu.cn)}

		\mdseries\normalsize
		\bigskip
		Jinrong Wang, Xiaoqing Yue$^*$
		
		\footnotesize
		\smallskip
		School of Mathematical Sciences, Key Laboratory of Intelligent Computing and  Applications (Ministry of Education), Tongji University, Shanghai 200092, China
		
		\smallskip
		E-mails: 2130926@tongji.edu.cn, xiaoqingyue@tongji.edu.cn 
	\end{center}
	{
		\footnotesize
		\noindent\textbf{Abstract}. In this paper, Lie conformal superalgebras of rank $(2+1)$ are completely classified (up to isomorphism) and their automorphism groups are determined. Furthermore, we give the classification of the finite irreducible conformal modules over them and the actions are explicitly described.
		
		
	}
	
	\section{Introduction}\label{sect:intro}
	\hspace{1.5em}The notion of Lie conformal (super)algebras was originally introduced by Kac in \cite{K98}. Lie conformal superalgebras axiomatize the singular part of the operator product expansion of chiral fields in conformal field theory. Moreover, it has been shown that the theory of Lie conformal superalgebras is an adequate tool for the study of Lie superalgebras satisfying the locality property \cite{K97}. In the past two decades, the structure and representation theories of Lie conformal superalgerbras have been greatly developed \cite{CK,CKW,FK,FKR,K98,WWX,WY,X22,ZCY,MZ}.
	
	As an algebraic object, the problem of classification of the Lie conformal (super)algebras is remarkable. Finite simple Lie conformal superalgebras were classified completely in \cite{FK}. The list consists of current Lie conformal superalgebras $\mathrm{Cur}\,\mathfrak{g}$ associated to a simple finite-dimensional Lie superalgebra $\mathfrak{g}$ and six series of Lie conformal superalgebras called Cartan type. The classifications of their finite irreducible conformal modules were given in a series of papers \cite{BKL10,BKL13,BKLR,CK,MZ}. 
	
	However, when it comes to the non-simple cases, it is much more difficult, partially due to that there does not exist conformal analogues of Levis' theorem. The rank one Lie conformal algebras were classified in \cite{DK}, which shows that they are isomorphic to either a commutative Lie conformal algebra or the Virasoro Lie conformal algebra 
	\begin{equation*}
		\Vir=\cp L\ \text{ with }\lambda\text{-bracket }\ [L_\lambda L]=(\partial+2\lambda)L.
	\end{equation*}
	Therefore, the Virasoro Lie conformal algebra is the simplest nontrivial Lie conformal algebra and the representation theory of $\Vir$ can be applied to any non-simple Lie conformal algebra which contains it as a subalgebra \cite{CK}. The rank $(1+1)$ Lie conformal superalgebras were classified in \cite{ZCY} and the most important one is the Neveu--Schwarz Lie conformal superalgebra (see Example~\ref{exa:NS}).
	
	Classification of the remaining higher rank Lie conformal (super)algebras are challengeable. In general, the classification of Lie conformal (super)algebras is based on solving the equations of two variable polynomials (see \eqref{K.1} and \eqref{K.2}). Kac commented in \cite{K98} that ``It is clearly impossible to solve those equations directly for $n\ge 2$." Encouragingly, the rank two Lie conformal algebras were classified completely in \cite{BCH} recently, which is still based on solving equations. Then, the finite irreducible conformal modules of rank two Lie conformal algebras were classified in \cite{XHW}. Besides, many rank three Lie conformal algebras were constructed, such as Schr\"odinger--Virasoro Lie conformal algebra (cf. \cite{SY}) and $\mathcal{W}(a,b,r)$ Lie conformal algebra (cf. \cite{LXY}). In particular, Leibniz conformal algebras of rank three and the Schr\"odinger--Virasoro type Lie conformal algebras were classified completely in \cite{W} and \cite{WW}, respectively.
	
	Since the odd part of a Lie conformal superalgebra can be viewed as a conformal module over its even part, in this paper, we classify the Lie conformal superalgebra of rank $(2+1)$ based on the aforementioned conclusions. In summary, we classify them into 18 cases, of which 5 are trivial (see Theorem~\ref{thm:classification}). The trivial case means the Lie conformal superalgebra satisfies either of the following two conditions:
	\begin{itemize}[itemsep=0pt]
		\item The $\lambda$-brackets containing the odd part is trivial (see Remark~\ref{trivial case}).
		\item The $\lambda$-brackets containing the even part is trivial and its rank is $(n+1)$, for some positive integer $n$ (see Proposition~\ref{prop:O}).
	\end{itemize}
	For the nontrivial cases, which are more important and our focuses, there are 4 solvable Lie conformal superalgebras, 5 Lie conformal superalgebras that can be decomposed into the direct sum of two Lie conformal (super)algebras and the last 4 Lie conformal superalgebras may be more noteworthy. For instance, one Lie conformal superalgebra of rank $(2+1)$ with a basis $\{A,B,X\}$ satisfying the relations 
	\begin{align*}
		[A_\lambda A]=(\partial+2\lambda)A,\ [A_\lambda B]=(\partial+a\lambda+b)B,\ [A_\lambda X]=(\partial+\frac{a+1}{2}\lambda+\frac{b}{2})X,\ [X_\lambda X]=2B,
	\end{align*}
	is an important example with two parameters $a,b\in\bC$. Recall that there exists a close correspondence between the Lie conformal superalgebra and a class of infinite-dimensional Lie superalgebra. If we take $a=1,b=0$, then it is related to the super Heisenberg--Virasoro Lie algebra of Ramond type or Neveu--Schwarz type \cite{CDH}. If we take $a=2,b=0$, then it associates to the $N=1$ super Bondi--Metzner--Sachs ($\mathrm{BMS}_3$) algebra, which is used to describe the asymptotic structure of the $N=1$ supergravity \cite{BDMT,DGL}. 
	
	In addition, we are curious whether the following Lie conformal superalgebra has been or can be applied to the fields of Lie superalgebra or mathematical physics in the future. It is a Lie conformal superalgebra with a basis $\{A,B,X\}$ satisfying the relations 
	\begin{align*}
		[A_\lambda A]=(\partial+2\lambda)A,\ [A_\lambda B]=(\partial+b)B,\ [A_\lambda X]=(\partial+\lambda+\frac{b}{2})X,\ [X_\lambda X]=(\partial+b)B,
	\end{align*}
	for some $b\in\bC$.
	
	Furthermore, in this paper, the finite nontrivial irreducible conformal modules (FNICMs) of rank $(2+1)$ Lie conformal superalgebras are also classified. The super analogue of Lie's theorem (cf. \cite{K77}) and Cheng--Kac's lemma (cf. \cite{CK}) play key roles in our proof. In conclusion, we construct five rank one conformal modules and three rank $(1+1)$ conformal modules (check Theorem~\ref{thm:FNICM} for more details).
	\begin{table}[!ht]
		\centering
		\begin{threeparttable}
			\begin{tabular}{ccc}
				\toprule
				\phantom{12}rank\phantom{12} & \phantom{123}$n$-parameters\phantom{123} & \phantom{12345}FNICM\phantom{12345} \\
				\midrule
				\multirow{3}*{$1$} &  $\bC[\lambda]$\tnote{*} & $N_\mathfrak{f}, N_\mathfrak{g}$ \\
				&  $n=2$ & $M_{A,\Delta,\eta},M_{B,\Delta,\eta}$ \\
				& $n=3$ & $M_{\Delta,\eta,\omega}$\\
				\hline
				\multirow{2}*{$(1+1)$} & $n=2$ & $M^+_{\Delta,\eta},M^-_{\Delta,\eta}$ \\
				& $n=4$ & $M_{\Delta,\zeta,c,\epsilon}$\\
				\bottomrule
			\end{tabular}
			\begin{tablenotes}
				\footnotesize
				\item[*] $N_\mathfrak{f}$, $N_\mathfrak{g}$ are determined by a nonzero polynomial respectively.
			\end{tablenotes}
		\end{threeparttable}
	\end{table}

	This paper is organized as follows. In Section 2, we first review some basic definitions of Lie conformal superalgebras and conformal modules. Then, we introduce two important conformal modules, which are the Virasoro module and the Neveu--Schwarz module.
	
	In Section 3, based on the conclusions of the classification of rank two Lie conformal algebras and their finite conformal modules, we classify the Lie conformal superalgebra of rank $(2+1)$ into 4 types and 18 cases. 
	
	In Section 4, the automorphism groups of the nontrivial rank $(2+1)$ Lie conformal superalgebras are investigated, which is an extension of the automorphism groups of the rank two Lie conformal algebras. 
	
	In the last section, the finite irreducible conformal modules over the nontrivial Lie conformal superalgebras of rank $(2+1)$ are classified completely. The conclusion shows that their finite conformal modules must be of rank one or rank $(1+1)$. As an application, we discuss a class of Lie conformal superalgebras and their irreducible conformal modules. 
	
	Throughout this paper, we use notations $\bC$, $\bCx$, $\bZ$ and $\bZ_+$ to represent the set of complex numbers, nonzero complex numbers, integers and nonnegative integers, respectively. In addition, all vector spaces and tensor products are over $\bC$. For any vector space $V$, we use $V[\lambda]$ to denote the set of polynomials of $\lambda$ with coefficients in $V$. We use $(A,B)=(C,D)$ to represent $A=C$ and $B=D$ for convenience. Conversely, $(A,B)\neq(C,D)$ means $A\neq C$ or $B\neq D$. 
	
	\section{Preliminaries}\label{sect:prelim}
	\hspace{1.5em}In this section, we will introduce some basic definitions, notations and related results about Lie conformal superalgebras and their modules.
	
	Let $V$ be a superspace that is a \bZz-graded linear space with a direct sum $V=V_\ep \oplus V_\op$. If $x\in V_\theta$, $\theta\in\bZz=\{\ep,\op\}$, then we say $x$ is \textit{homogeneous} and of \textit{parity $\theta$}. The parity of a homogeneous element $x$ is denoted by $|x|$. Throughout what follows, if $|x|$ occurs in an expression, then it is assumed that $x$ is homogeneous and that the expression extends to the other elements by linearity. 
	
	\subsection{Lie conformal superalgebra}
	\begin{definition}\label{def:lie conformal superalgebra}
		A \textit{Lie conformal superalgebra} $R=R_\ep \oplus R_\op$ is a \bZz-graded $\cp$-module with a $\bC$-linear map $R\otimes R\rightarrow \bC[\lambda]\otimes R$, $a\otimes b\mapsto [a_\lambda b]$ called $\lambda$-bracket, and satisfying the following axioms $(a,b,c\in R)$:
		\begin{align}
			\text{(conformal sesquilinearity)}&\phantom{1234} [\partial a_\lambda b]=-\lambda[a_\lambda b],\ [a_\lambda \partial b]=(\partial+\lambda)[a_\lambda b],\label{conformal-sesqui}\\
			\text{(skew-symmetry)}&\phantom{1234} [a_\lambda b]=-(-1)^{|a||b|}[b_{-\lambda-\partial}a],\label{skew-symm}\\
			\text{(Jacobi identity)}&\phantom{1234} [a_\lambda[b_\mu c]]=[[a_\lambda b]_{\lambda+\mu}c]+(-1)^{|a||b|}[b_\mu[a_\lambda c]].\label{jocobi-identity}
		\end{align}
	\end{definition}
	
	For a Lie conformal superalgebra $R=R_\ep\oplus R_\op$ and for any $\theta,\xi\in\{\ep,\op\}$, $R_\theta R_\xi\subseteq R_{\theta+\xi}$. $R$ is called \textit{finite} if it is finitely generated over $\bC[\partial]$. If it is further a free $\bC[\partial]$-module, then the \emph{rank} of $R$ is defined as the rank of the underlying free $\bC[\partial]$-module. Sometimes the rank is written as the sum of two numbers which means the sum of the rank of $R_\ep$ and the rank of $R_\op$.
	
	\begin{definition}\label{homo}
		Let $R, \bar{R}$ be two Lie conformal superalgebras. A graded-preserved linear map $\sigma:R\rightarrow \bar{R}$ is a \textit{homomorphism} of Lie conformal superalgebras if $\sigma$ satisfies $\sigma\partial=\partial\sigma$ and $\sigma([a_\lambda b])=[\sigma(a)_\lambda\sigma(b)]$, for any $a,b\in R$. The homomorphism
		is called an \textit{isomorphism} if it is further a bijection. For the Lie conformal superalgebra $R$, the set of all isomorphisms $R\rightarrow R$ is called \textit{automorphism group} under the composition of isomorphisms and denoted by $\Aut(R)$.
	\end{definition}
	
	The notions of subalgebra, ideal and abelian ideal of a Lie conformal superalgebra are defined in the usual way.
	
	\begin{example}\label{exa:NS}\textup{(\cite{CK}) (Neveu--Schwarz Lie conformal superalgebra)} Let $\mathcal{NS}=\bC[\partial]L\oplus \bC[\partial]G$ be a free \bZz-graded $\bC[\partial]$-module. Define
		\begin{align*}
			[L_\lambda L]=(\partial+2\lambda)L,\ \ [L_\lambda G]=(\partial+\frac32\lambda)G,\ \ [G_\lambda G]=2L,
		\end{align*}
		where $\mathcal{NS}_\ep=\bC[\partial]L$ and $\mathcal{NS}_\op=\bC[\partial]G$. Then $\mathcal{NS}$ is a Lie conformal superalgebra of rank $(1+1)$. We call it the Neveu--Schwarz Lie conformal superalgebra.
	\end{example}
	
	As a matter of fact, if we take $|a|=|b|=\ep$ in Definition~\ref{def:lie conformal superalgebra}, we get the definition of the usual Lie conformal algebra. In other words, for a lie conformal superalgebra $R=R_\ep\oplus R_\op$, $R_\ep$ is a usual Lie conformal algebra. Besides the Virasoro Lie conformal algebra, we give another important example.
	\begin{example}\label{HV}
		The Heisenberg--Virasoro Lie conformal algebra denoted by $\mathcal{HV}$ is a rank two Lie conformal algebra, which has a free $\bC[\partial]$-basis $\{L,H\}$ such that\textup{:}
		\begin{align*}
			[L_\lambda L]=(\partial+2\lambda)L,\ \ [L_\lambda H]=(\partial+\lambda)H,\ \ [H_\lambda H]=0.
		\end{align*}
		It was firstly introduced in \textup{\cite{SY}}, where its structure was discussed.
	\end{example}
	
	Suppose $R$ is a Lie conformal superalgebra. For any $a,b\in R$, we write 
	\begin{align}\label{jth prdouct}
		[a_\lambda b]=\sum_{j\in\bZ_+}\lambda^{(j)}a_{(j)}b\ \text{ with }\ \lambda^{(j)}=\frac{\lambda^j}{j!}.
	\end{align}
	For every $j\in\bZ_+$, we have the \bC-linear map: $R\otimes R\rightarrow R,\ a\otimes b\mapsto a_{(j)}b$, which is called the $j$-\textit{th product}. The \textit{derived algebra} of a Lie conformal superalgebra $R$ is the vector space $R'=\Span_\bC\{a_{(n)}b\,|\,a,b\in R,\ n\in\bZ_+\}$. It is easy to check that $R'$ is an ideal of $R$. Define $R^{(1)}=R'$ and $R^{(n+1)}=R^{(n)'}$ for any $n\ge 1$. Then we get the \textit{derived series}
	\begin{align*}
		R=R^{(0)}\supseteq R^{(1)}\supseteq\cdots\supseteq R^{(n)}\supseteq R^{(n+1)}\supseteq\cdots
	\end{align*}
	of $R$. If there exists $N\in\bZ_+$ such that $R^{(n)}=0$ for any $n\ge N$, then $R$ is said to be \textit{solvable}. Let $R^n=\Span_\bC\{a_{(m_1)}^1a_{(m_2)}^2\cdots a_{(m_{n-1})}^{n-1}a^n\,|\,a^i\in R,\ m_i\in\bZ_+\}$, then we have $R^{(n)}\subseteq R^n$ and the following series of ideals
	\begin{align*}
		R=R^{0}\supseteq R^{1}\supseteq\cdots \supseteq R^n\supseteq R^{n+1}\supseteq\cdots.
	\end{align*}
	If there exists $N\in\bZ_+$ such that $R^{n}=0$ for any $n\ge N$, then $R$ is said to be \textit{nilpotent}.
	
	A Lie conformal superalgebra is \textit{simple} if it is non-abelian and contains no nontrivial ideals. If one Lie conformal superalgebra has no nonzero abelian ideals, then it is \textit{semisimple}. Since the second last term of the derived series of a solvable Lie conformal superalgebra is an abelian ideal, this is equivalent to saying that a semisimple Lie conformal superalgebra has no nonzero solvable ideals. 
	
	Let $R=\bigoplus_{j=1}^n\cp X^j$ be a \bZz-graded \cp-module with $|X^j|$ denoted by $|j|$, and let $[X^i_\lambda X^j]=\sum_k Q_{ij}^{k}(\partial,\lambda)X^k$, where $Q_{ij}^{k}(\partial,\lambda)\in\bC[\partial,\lambda]$, $i,j,k=1,\dots,n$. These $\lambda$-brackets give rise to a structure of a Lie conformal superalgebra if $Q_{ij}^{k}(\partial,\lambda)$ satisfies the following relations which are equivalent to axioms \eqref{skew-symm} and \eqref{jocobi-identity} respectively (cf. \cite{K98}):
	\begin{small}
		\begin{align}
			Q_{ij}^{k}(\partial,\lambda)&=-(-1)^{|i||j|}Q_{ji}^{k}(\partial,-\partial-\lambda),\label{K.1}\\
			\sum_{s=1}^n\big(Q_{jk}^{s}(\partial+\lambda,\mu)Q_{is}^{t}(\partial,\lambda)-(-1)^{|i||j|}Q_{ik}^{s}&(\partial+\mu,\lambda)Q_{js}^{t}(\partial,\mu)\big)=\sum_{s=1}^nQ_{ij}^{s}(-\lambda-\mu,\lambda)Q_{sk}^{t}(\partial,\lambda+\mu).\label{K.2}
		\end{align}
	\end{small}

	Therefore, the description of a Lie conformal superalgebra with free \cp-modules of finite rank can be reduced to the solution of a finite system of polynomial equations in two indeterminates. 
	\begin{definition}
		A polynomial $f(x,y)\in\bC[x,y]$ is said to be \textit{skew-symmetric} if it satisfies the equality
		$$f(x,y)=-f(x,-x-y).$$
	\end{definition}
	By \eqref{K.1}, $Q_{ii}^{k}(\partial,\lambda)$ are all skew-symmetric when $|i|=0$, which play important roles in the classification of Lie conformal algebras (cf. \cite{BCH}). However, if $|i|=1$, we have $Q_{ii}^k(\partial,\lambda)=Q_{ii}^k(\partial,-\partial-\lambda)$. It implies that $Q_{ii}^k(\partial,\lambda)$ is independent of the second variable, which is very useful for our later discussion.

	\subsection{Conformal module}
	\begin{definition}\label{def:conformal module}
		A \textit{conformal module} $M=M_\ep+M_\op$ over a Lie conformal superalgebra $R$ is a \bZz-graded \cp-module equipped with a \bC-linear map $R\otimes M\rightarrow\bC[\lambda]\otimes M$, $a\otimes v\mapsto a_\lambda v$ called $\lambda$-action, and satisfying the following relations $(a,b\in R, v\in M)$:
		\begin{align}
			(\partial a)_\lambda v=-\lambda a_\lambda v,\ a_\lambda(\partial v)=(\partial+\lambda)a_\lambda v,\label{conmod.1}\\
			[a_\lambda b]_{\lambda+\mu}v=a_\lambda(b_\mu v)-(-1)^{|a||b|}b_\mu(a_\lambda v).\label{conmod.2}
		\end{align}
	\end{definition}
	Let $R$ be a Lie conformal superalgebra. A conformal $R$-module $M=M_\ep+M_\op$ is \textit{finite} if it is finitely generated as a \cp-module. We say that $M$ has \textit{rank} $(m+n)$, if $M_\ep$ has rank $m$ and $M_\op$ has rank $n$ as \cp-modules. $M$ is called to be \textit{irreducible}, if it has no nontrivial proper submodule.
	
	Similarly, if we take $|a|=|b|=\ep$ in Definition~\ref{def:conformal module}, we get the definition of the conformal module over the usual Lie conformal algebra. Furthermore, $R_\op$ is an $R_\ep$-module, if we define the action $(\mathrm{ad}\ a)_\lambda:\, a\mapsto[a_\lambda b]$, for any $a\in R_\ep$, $b\in R_\op$. 
	
	Obviously, for a fixed complex number $\tau$, the \cp-module $\bC c_\tau$ with $\partial c_\tau=\tau c_\tau$, $R_\lambda c_\tau=0$ is a conformal $R$-module. We refer to $\bC c_\tau$ as the \textit{even} (resp., \textit{odd}) \textit{one-dimensional trivial module} if $|c_\tau|=\ep$ (resp., $|c_\tau|=\op$). They exhaust all trivial irreducible conformal $R$-modules. Thus, we only need to consider the nontrivial cases.
	
	Let $M$ be a conformal $R$-module. An element $m\in M$ is called a \textit{torsion element} if there exists a nonzero polynomial $p(\partial)\in\cp$ such that $p(\partial)m=0$. From \cite{DK,FKR}, we have the following lemma.
	\begin{lemma}\label{lem:free.finite.rank.mod}
		Let $R$ be a Lie conformal algebra and $M$ be a finite nontrivial irreducible conformal $R$-module. Then $M$ has no nonzero torsion elements and is free of finite rank as a $\cp$-module.
	\end{lemma}
	
	Now, we introduce two important examples of nontrivial conformal module.
	
	\begin{proposition}\label{vir.mod}\textup{(\cite{CK})}
		For the Virasoro Lie conformal algebra $\mathrm{Vir}$, all the free nontrivial $\mathrm{Vir}$-modules of rank one over $\bC[\partial]$ are the following ones $(\Delta,a\in\bC)$\textup{:}
		\begin{align*}
			V_{\Delta,a}=\bC[\partial]v,\ \ L_\lambda v=(\partial+\Delta\lambda+a)v.
		\end{align*}
		The module $V_{\Delta,a}$ is irreducible if and only if $\Delta\neq 0$. The module $V_{0,a}$ contains a unique nontrivial submodule $(\partial+a)V_{0,\alpha}$ isomorphic to $V_{1,a}$. Moreover, the modules $V_{\Delta,a}$ with $\Delta\neq 0$ exhaust all finite  nontrivial irreducible $\mathrm{Vir}$-modules.
	\end{proposition}
	
	\begin{proposition}\label{ns.mod}\textup{(\cite{CK,CKW})}  
		For the Neveu--Schwarz Lie conformal superalgebra $\mathcal{NS}$, the finite nontrivial irreducible conformal module over it is isomorphic to one of the following\textup{:} 
		
		\subno{1} $\bar{V}_{\Delta,a}=\cp v_\ep\oplus\cp v_\op$ with
		\begin{alignat*}{2}
			&L_\lambda v_\ep=(\partial+\Delta\lambda+a)v_\ep,\ \ &L_\lambda v_\op&=\big(\partial+(\Delta+\frac{1}{2})\lambda+a\big)v_\op,\\
			&G_\lambda v_\ep=v_\op, &G_\lambda v_\op&=(\partial+2\Delta\lambda+a)v_\ep,
		\end{alignat*}
		for some $\Delta\in\bCx$ and $a\in\bC$.
		
		\subno{2} $\bar{V}'_{\Delta',a'}=\cp v'_\ep\oplus\cp v'_\op$ with
		\begin{alignat*}{3}
			&L_\lambda v'_\ep=(\partial+\Delta'\lambda+a')v'_\ep, &L_\lambda v'_\op&=\big(\partial+(\Delta'-\frac{1}{2})\lambda+a'\big)v'_\op,\\
			&G_\lambda v'_\ep=\big(\partial+(2\Delta'-1)\lambda+a'\big)v'_\op,\ \ \  &G_\lambda v'_\op&=v'_\ep,
		\end{alignat*}
		for some $\frac{1}{2}\neq\Delta'\in\bC$ and $a'\in\bC$.
	\end{proposition}
	
	\begin{definition}\label{def:annalg}
		The \textit{annihilation superalgebra} $\Lie(R)^+$ of a Lie conformal superalgebra $R$ is a Lie superalgebra with \bC-basis $\{a_{(n)}\ |\ a\in R,\ n\in\bZ_+\}$ and relations
		\begin{align*}
			(\partial a)_{(n)}=-na_{(n-1)},\ \ (a+b)_{(n)}=a_{(n)}+b_{(n)},\ \ (ka)_{(n)}=ka_{(n)},
		\end{align*}
		for $a,b\in R$ and $k\in\bC$, and the Lie brackets of $\Lie(R)^+$ are given by
		\begin{equation}\label{ann.lie.bracket}
			[a_{(m)},b_{(n)}]=\sum_{j\in\bZ_+}\binom{m}{j}(a_{(j)}b)_{(m+n-j)}.
		\end{equation}
	\end{definition}
	
	The parity $|a_{(n)}|$ of $a_{(n)}\in\Lie(R)^+$ is the same as $|a|$ for any homogeneous $a\in R$ and $n\in\bZ_+$. The \textit{extended annihilation superalgebra} $\Lie(R)^e$ is defined by the semidirect sum
	\begin{equation}\label{extended.ann.lie}
		\Lie(R)^e=\bC\partial\ltimes\Lie(R)^+\ \text{ with }\ [\partial,a_{(n)}]=-na_{(n-1)}.
	\end{equation}

	\begin{example}
		For the Virasoro Lie conformal algebra $\Vir$, define $L_n=L_{(n+1)}$ for $n\ge -1$. Then $\Lie(\Vir)^+$ is a Lie algebra with basis $\{L_n\ |\ n\ge -1\}$ and relations
		\begin{align*}
			[L_m,L_n]=(m-n)L_{m+n},
		\end{align*}
		for $m,n\ge -1$. In addition, $\Lie(\Vir)^+$ is isomorphic to the Lie algebra of regular vector fields over \bC.
	\end{example}
	
	Similar to the definition of $j$-th product $a_{(j)}b$ of two elements $a,b\in R$, we can define $j$-\textit{th actions} of $R$ over $M$ for all $a\in R$, $v\in M$ by
	\begin{align}\label{jth.action}
		a_\lambda v=\sum_{j\in\bZ_+}\lambda^{(j)}a_{(j)}v\ \text{ with }\ \lambda^{(j)}=\frac{\lambda^j}{j!}.
	\end{align}
	
	If $M$ is a conformal $R$-module, by \eqref{jth.action}, it can be regarded as a module over $\Lie(R)^e$ naturally. Conversely, if for any $v$ belongs to the $\Lie(R)^e$-module $M$ and $a\in R$, there exists $N(a,v)\in\bZ_+$ such that $a_{(n)}v=0$ for $n>N(a,v)$, then $M$ can be viewed as an $R$-module. By abuse of notations, this module is said to be \textit{conformal}. We have the following conclusion:
	\begin{proposition}\label{prop:conformal}\textup{(\cite{CK})}
		A conformal module $M$ over a Lie conformal superalgebra $R$ is the same as a module over the Lie superalgebra $\Lie(R)^e$ satisfying $a_{(n)} v=0$ for $a\in R$, $v\in M$, $n\gg 0$, namely it is conformal.
	\end{proposition}

	\section{Lie conformal superalgebra of rank $(2+1)$}\label{sect:(2+1)}
	\hspace{1.5em}In this section, we will classify the Lie conformal superalgebras of rank $(2+1)$ completely. Let $R=R_\ep\oplus R_\op$ be a free Lie conformal superalgebra of rank $(2+1)$ in this section. The even part $R_\ep$ is a usual Lie conformal algebra of rank two, which has been classified completely in \cite{BCH}. We use $\CH$ to denote a free rank two Lie conformal algebra in the sequel. For convenience, we summarize the classification results into four types.
	\begin{proposition}\label{rank2.lie.con.alg}\textup{(\cite{BCH})} Let $\CH$ be a free rank two Lie conformal algebra. Then, up to isomorphism, $\CH$ is one of the following types.
		
		\subno{1} \textup{\textbf{Type A}}. If $\CH$ is solvable, then there is a basis $\{A,B\}$ such that
		\begin{align}\label{H.A}
			[B_\lambda B]=0,\ [A_\lambda B]=P_1(\partial,\lambda)B,\ [A_\lambda A]=Q_1(\partial,\lambda)B,
		\end{align}
		for some $P_1(\partial,\lambda),Q_1(\partial,\lambda)\in\bC[\partial,\lambda]$. More explicitly, if $\CH$ is nilpotent, then $P_1(\partial,\lambda)=0$ and $Q_1(\partial,\lambda)$ is any skew-symmetric polynomial. If $\CH$ is solvable but not nilpotent, then $P_1(\partial,\lambda)=p(\lambda)$ and $Q_1(\partial,\lambda)=0$, where $p(\lambda)$ is a nonzero polynomial.
		
		\subno{2} \textup{\textbf{Type B}}. If $\CH$ is the direct sum of two Virasoro Lie conformal algebras, then there exists a basis $\{A,B\}$ such that
		\begin{align}\label{H.B}
			[A_\lambda A]=(\partial+2\lambda)A,\ [B_\lambda B]=(\partial+2\lambda)B,\ [A_\lambda B]=0.
		\end{align}
		
		\subno{3} \textup{\textbf{Type C}}. If $\CH$ is the direct sum of the rank one commutative Lie conformal algebra and the Virasoro Lie conformal algebra, then there exists a basis $\{A,B\}$ such that
		\begin{align}\label{H.C}
			[A_\lambda A]=(\partial+2\lambda)A,\ [A_\lambda B]=[B_\lambda B]=0.
		\end{align}
		
		If $\CH$ is none of the above, then we have the following type.
		
		\subno{4} \textup{\textbf{Type D}}. $\CH$ has a basis $\{A,B\}$ with the relations\textup{:}
		\begin{equation}\label{H.D}
			\begin{split}
				[A_\lambda A]=(\partial+2\lambda)A+Q(\partial,\lambda)B,\ [A_\lambda B]=(\partial+a\lambda+b)B,\ [B_\lambda B]=0,
			\end{split}
		\end{equation}
		where $a,b\in\bC$ and $Q(\partial,\lambda)\in\bC[\partial,\lambda]$ satisfying \textup{\textbf{Condition 1}} \textup{(}abbr. \textup{\textbf{C1})}.
		
		\textup{\textbf{Condition 1}:} $Q(\partial,\lambda)=0$ if $b\neq 0$. Otherwise, we have
		
		\begin{table}[htpb]
			\centering
			\begin{tabular}{|c|c|}    
				\hline
				\phantom{1}$a\in\bC$\phantom{1} & \phantom{01234567890}$Q(\partial,\lambda),c, d\in\bC$\phantom{01234567890}\\
				\hline
				$1$ & $c(\partial+2\lambda)$ \\
				\hline
				$0$ & $(\partial+2\lambda)\big(c(\partial\lambda+\lambda^2)+d\partial\big)$\\
				\hline
				$-1$ & $(\partial+2\lambda)\big(c\partial^2+d(\partial\lambda+\lambda^2)\partial\big)$\\
				\hline
				$-4$ & $c(\partial+2\lambda)(\partial\lambda+\lambda^2)^3$\\
				\hline
				$-6$ & $c(\partial+2\lambda)\big(11(\partial\lambda+\lambda^2)^4+2(\partial\lambda+\lambda^2)^3\partial^2\big)$\\
				\hline
			\end{tabular}
		\end{table}
	\end{proposition}
	
	Furthermore, since the odd part $R_\op$ is an $R_\ep$-module (i.e. $\CH$-module), we also need the conclusions of the conformal modules over rank two Lie conformal algebras which have been shown in \cite{XHW}. 
	\begin{proposition}\label{rank2.lie.con.alg.mod}\textup{(\cite{XHW})} Let $\CH=\cp A\oplus\cp B$ be a rank two Lie conformal algebra and $M=\cp v$ be the nontrivial conformal $\CH$-module of rank one, then the action of $\CH$ on $M$ has to be the following cases corresponding to the above four types\textup{:}
		
		\subno{1} \textup{\textbf{Type A}}. If $\CH$ is solvable with the relation \eqref{H.A}, then we have 
		\begin{align}\label{H.A.mod}
			A_\lambda v=\phi_A(\lambda)v,\ B_\lambda v=\phi_B(\lambda)v,
		\end{align}
		where $\phi_A(\lambda),\phi_B(\lambda)\in\bC[\lambda]$ are not zero simultaneously. In addition, if $(P_1(\partial,\lambda),Q_1(\partial,\lambda))\neq(0,0)$, then $\phi_B(\lambda)=0$.
		
		\subno{2} \textup{\textbf{Type B}}. If $\CH$ is semisimple with the relation \eqref{H.B}, then either
		\begin{align}
			A_\lambda v=(\partial+\alpha_1\lambda+\beta_1)v,\ B_\lambda v=0,\ \text{for some }\alpha_1,\beta_1\in\bC,\label{H.B.mod}
			\shortintertext{or}
			A_\lambda v=0,\ B_\lambda v=(\partial+\alpha_2\lambda+\beta_2)v,\ \text{for some }\alpha_2,\beta_2\in\bC.\label{H.B.mod'}
		\end{align}
		
		\subno{3} \textup{\textbf{Type C}}. If $\CH$ satisfies the relation \eqref{H.C}, then either
		\begin{align}
			A_\lambda v=(\partial+\alpha\lambda+\beta)v,\ B_\lambda v=0,\ \text{for some }\alpha,\beta\in\bC,\label{H.C.mod.1}
			\shortintertext{or}
			A_\lambda v=0,\ B_\lambda v=\phi(\lambda)v,\ \text{for some nonzero }\phi(\lambda)\in\bC[\lambda].\label{H.C.mod.2}
		\end{align}
		
		\subno{4} \textup{\textbf{Type D}}. If $\CH$ satisfies the relation \eqref{H.D}, then
		\begin{align}\label{H.D.mod}
			A_\lambda v=(\partial+\alpha\lambda+\beta)v,\ B_\lambda v=\gamma v,
		\end{align}
		where $\alpha,\beta,\gamma\in\bC$ such that if $(a,b,Q(\partial,\lambda))\neq(1,0,0)$, namely $\mathcal{H}$ is not the Heisenberg--Virasoro Lie conformal algebra, then $\gamma=0$.
	\end{proposition}
	
	The irreducibility of the conformal modules constructed in Proposition~\ref{rank2.lie.con.alg.mod} are also discussed in \cite{XHW}. To be concise, we only present the partial conclusions as follows:
	\begin{proposition}\label{rank2.lie.rank1.mod}\textup{(\cite{XHW})}
		Any finite nontrivial irreducible conformal modules over the rank two Lie conformal algebras is free of rank one.
	\end{proposition}
	
	Now, let $R=\cp A\oplus\cp B\oplus\cp X$ and $R_\ep=\CH=\cp A\oplus\cp B$, $R_\op=\cp X$. By Definition~\ref{def:lie conformal superalgebra}, we can further assume
	\begin{align*}
		[A_\lambda X]=\phi(\partial,\lambda)X,\ [B_\lambda X]=\varphi(\partial,\lambda)X,\ 
		[X_\lambda X]=\psi_1(\partial,\lambda)A+\psi_2(\partial,\lambda)B,
	\end{align*}
	where $\phi(\partial,\lambda),\varphi(\partial,\lambda),\psi_1(\partial,\lambda)$ and $\psi_2(\partial,\lambda)$ are all in $\bC[\partial,\lambda]$. Due to the skew-symmetry in \eqref{skew-symm}, we have $[X_\lambda X]=[X_{-\partial-\lambda} X]$. It leads to $\psi_i(\partial,\lambda)=\psi_i(\partial,-\partial-\lambda)$, where $i=1,2$. Considering the degree in $\partial$ on the both sides of the two equalities respectively, we can get $\psi_i(\partial,\lambda)=\psi_i(\partial)$. It follows that 
	\begin{align}\label{R.rel.origin}
		[A_\lambda X]=\phi(\partial,\lambda)X,\ [B_\lambda X]=\varphi(\partial,\lambda)X,\ 
		[X_\lambda X]=\psi_1(\partial)A+\psi_2(\partial)B.
	\end{align}
	\begin{remark}\label{trivial case}(Trivial case) 
		If $R_\ep$ is a Lie conformal algebra of type $A,B,C$ or $D$ and $[A_\lambda X]=[B_\lambda X]=[X_\lambda X]=0$, then $R$ must be a Lie conformal superalgebra. 
	\end{remark}
	
	Before we give the classification of the Lie conformal superalgebras of rank $(2+1)$, we have to develop the following two lemmas about polynomials.
	\begin{lemma}\label{lem:f(x+y)g(x,y)=f(x)h(y)}
		Let $f(x),h(x)\in\bC[x]$, $g(x,y)\in\bC[x,y]$ be nonzero polynomials. If $f(x+y)g(x,y)=f(x)h(y)$, then $f(x)=c$, $g(x,y)=h(y)$ where $c\in\bCx$ is a constant.
	\end{lemma}
	\begin{proof}
		It is not difficult to see that $\deg f(x+y)=\deg f(x)$ and $\deg f(x+y)+\deg_xg(x,y)=\deg f(x)$. Thus $\deg_xg(x,y)=0$, namely $g(x,y)=g(y)\in\bC[y]$. Then it gives rise to $f(x+y)g(y)=f(x)h(y)$. Suppose that $f(x)=\sum_{i=0}^nf_ix^i$ for some $n\in\bZ_+$ and $f_n\neq 0$. Equating the coefficients of $x^n$, we obtain $g(y)f_n=h(y)f_n$. So $g(y)=h(y)$ and $f(x+y)=f(x)$. It follows that $f(x)=c\in\bCx$.
	\end{proof}
	
	\begin{lemma}\label{lem:x+ay+b)f(x+y)=(x+2a'y-y+2b')f(x)}
		Let $a,b,\alpha,\beta\in\bC$ and $f(x)$ be a nonzero polynomial. If $(x+ay+b)f(x+y)=(x+2\alpha y-y+2\beta)f(x)$, then $f(x)=c\in\bCx$ with $a=2\alpha-1$, $b=2\beta$ or $f(x)=k_1(x+b)$ for some $k_1\in\bCx$ with $a=0$, $\alpha=1$, $b=2\beta$.
	\end{lemma}
	\begin{proof}
		Equate the degree of $y$ on the both sides of the original equality, we deduce that $\deg f(x)\le1$. If $\deg f(x)=1$, we can assume that $f(x)=k_0+k_1x$ with $k_0\in\bC, k_1\in\bCx$. Then plug it into the original equality it follows that $a=0$, $\alpha=1$, $b=2\beta$ and $f(x)=k_1(x+b)$. If $\deg f(x)=0$, put $f(x)=c\in\bCx$ and take it into the original equality. We obtain that $a=2\alpha-1$ and $b=2\beta$.
	\end{proof}

	\subsection{Type $\mathcal{A}$}
	\hspace{1.5em}Let $R_\ep$ be solvable. Firstly, we discuss the extreme case. Let $R_\ep$ be the direct sum of two commutative Lie conformal algebras of rank one, namely $[A_\lambda A]=[A_\lambda B]=[B_\lambda B]=0$. Actually, we can generalize it to the case of higher rank. 
	
	Suppose that $R_\ep=\bigoplus_{i=1}^n\cp A_i$ with $[{A_i}_\lambda A_j]=0$ for any $i,j\in\{1,\dots,n\}$. To classify the Lie conformal superalgebra $R=R_\ep\oplus R_\op$, where $R_\op=\cp X$, we need to determine the free rank one conformal modules over $R_\ep$ firstly. For the finite solvable Lie conformal algebra, we have the following conformal analog of Lie's theorem:
	\begin{proposition}\label{lem:con.Lie.thm}\textup{(\cite{DK})}
		Let $R$ be a finite solvable Lie conformal algebra and $M$ be a finite conformal $R$-module. Then there exists $v\in M$ such that $a_\lambda v=\phi(a)v$, where $\phi: R\rightarrow \bC[\lambda]$ and $a\in R$. In particular, every finite nontrivial irreducible conformal $R$-module $M$ is free of rank one.
	\end{proposition}
	
	It is not difficult to get the following conclusion by Proposition~\ref{lem:con.Lie.thm}. 
	\begin{lemma}\label{lem:n.sol.mod}
		Let $M$ be the finite nontrivial conformal module over $R_\ep$. Then 
		\begin{equation*}
			M=\cp v,\ \ {A_i}_\lambda v=\phi_i(\lambda)v,\ \ \text{for all }\ i=1,\dots,n,
		\end{equation*}
		where $\phi_i(\lambda)\in\bC[\lambda]$ are not zero simultaneously.
	\end{lemma}
	
	Then, we can prove the following proposition.
	\begin{proposition}\label{prop:O}
		Let $R=R_\ep\oplus R_\op$ be a Lie conformal superalgebra with basis $\{A_i,X\ |\ i=1,\dots,n\}$, where $R_\ep=\bigoplus_{i=1}^n\cp A_i$ and $R_\op=\cp X$ with the relation $[{A_i}_\lambda A_j]=0$ for any $i,j\in\{1,\dots,n\}$. Then $R$ is isomorphic to one of the following Lie conformal superalgebras\textup{:}
		
		\subno{1} $\Tilde{{\mathcal{O}}_1}:=(\bigoplus_{i=1}^n\cp A_i)\oplus\cp X$,
		\begin{align*}
			[{A_i}_\lambda A_j]=0,\ [{A_i}_\lambda X]=\phi_i(\lambda)X,\ [X_\lambda X]=0,\ \forall\ i,j\in\{1,\dots,n\},\ \text{for some }\phi_i(\lambda)\in\bC[\lambda].
		\end{align*}
		
		\subno{2} $\Tilde{{\mathcal{O}}_2}:=(\bigoplus_{i=1}^n\cp A_i)\oplus\cp X$,
		\begin{align*}
			[{A_i}_\lambda A_j]=0,\ [{A_i}_\lambda X]=0,\ [X_\lambda X]=\sum_{i=1}^n\psi_i(\partial)A_i,\ \forall\ i,j\in\{1,\dots,n\},\ \text{for some }\psi_i(\partial)\in\cp.
		\end{align*}
	\end{proposition}
	\begin{proof}
		Given that $X$ is the module over $R_\ep$, by Lemma~\ref{lem:n.sol.mod}, we can assume that $[{A_i}_\lambda X]=\phi_i(\lambda)X$, $i=1,\dots,n$, for some $\phi_i(\lambda)\in\bC[\lambda]$. $\phi_i(\lambda)$ can be zero simultaneously here, because it can be viewed as a trivial conformal module. We further assume that $[X_\lambda X]=\sum_{i=1}^n\psi_i(\partial)A_i$, $i=1,\dots,n$, for some $\psi_i(\lambda)\in\bC[\lambda]$, since $[X_\lambda X]=[X_{-\partial-\lambda} X]$.
		
		Using the Jacobi identity $[{A_i}_\lambda[X_\mu X]]=[[{A_i}_\lambda X]_{\lambda+\mu}X]+[X_\mu[{A_i}_\lambda X]]$ and comparing the coefficients of $A_j$, we obtain  
		\begin{align}\label{general O}
			2\phi_i(\lambda)\psi_j(\partial)=0,\ \ \text{for all }\ i,j\in\{1,\dots,n\}. 
		\end{align}
		If $\psi_j(\partial)=0$ for any $j$, then we get $\Tilde{{\mathcal{O}}_1}$. Otherwise, without loss of generality, we can suppose $\psi_j(\partial)\neq0$ for some fixed $j\in\{1,\dots,n\}$. According to \eqref{general O}, we have $\phi_i(\lambda)=0$ for all $i$, then $\Tilde{\mathcal{O}}_2$ follows.
	\end{proof}
	We call the Lie conformal superalgebras constructed in Proposition~\ref{prop:O} trivial. In particular, we get the following conclusion.
	
	\begin{lemma}\label{lem:O}
		For a Lie conformal superalgebra $R=R_\ep\oplus R_\op$ of rank $(2+1)$, if $R_\ep$ is the direct sum of two commutative Lie conformal algebras of rank one, then $R$ is isomorphic to one of the following Lie conformal superalgebras\textup{:}
		
		\subno{1} $\mathcal{O}_1:=\cp A\oplus\cp B\oplus\cp X$,
		\begin{align*}
			[A_\lambda A]=[A_\lambda B]=[B_\lambda B]=0,\ [A_\lambda X]=\phi(\lambda)X,\ [B_\lambda X]=\varphi(\lambda)X,\ [X_\lambda X]=0,
		\end{align*}
		where $\phi(\lambda),\varphi(\lambda)\in\bC[\lambda]$.
		
		\subno{2} $\mathcal{O}_2:=\cp A\oplus\cp B\oplus\cp X$,
		\begin{align*}
			[A_\lambda A]=[A_\lambda B]=[B_\lambda B]=0,\ [A_\lambda X]=[B_\lambda X]=0,\ [X_\lambda X]=\psi_1(\partial)A+\psi_2(\partial)B,
		\end{align*}
		where $\psi_1(\partial),\psi_2(\partial)\in\cp$.
	\end{lemma}
	
	From now on, we suppose that the basis $\{A,B,X\}$ satisfies the relations \eqref{H.A} and \eqref{R.rel.origin}. Let at least one of $P_1(\partial,\lambda)$ and $Q_1(\partial,\lambda)$ be nonzero. By \eqref{H.A.mod}, we can assume that $[A_\lambda X]=\phi(\lambda)X$ for some $\phi(\lambda)\in\bC[\lambda]$ since its $R_\ep$-actions are independent of $\partial$. In addition, we can conclude that $[B_\lambda X]=0$ by Proposition~\ref{rank2.lie.con.alg.mod}(1).
	
	Using the Jacobi identity $[A_\lambda[X_\mu X]]=[[A_\lambda X]_{\lambda+\mu}X]+[X_\mu[A_\lambda X]]$ and comparing the coefficients of $A$ and $B$ respectively, we obtain 
	\begin{align}
		2\phi(\lambda)\psi_1(\partial)&=0,\label{A.AXX.A.0}\\
		2\phi(\lambda)\psi_2(\partial)&=\psi_1(\partial+\lambda)Q_1(\partial,\lambda)+\psi_2(\partial+\lambda)P_1(\partial,\lambda).\label{A.AXX.B.0}
	\end{align}
	Similarly, using the Jacobi identity $[B_\lambda[X_\mu X]]=[[B_\lambda X]_{\lambda+\mu}X]+[X_\mu[B_\lambda X]]$ and comparing the coefficients of $B$, we get
	\begin{align}
		\psi_1(\partial+\lambda)P_1(\partial,-\partial-\lambda)=0.\label{A.BXX.B.0}
	\end{align}
	To solve the equalities \eqref{A.AXX.A.0}--\eqref{A.BXX.B.0}, we can discuss the following two cases:
	
	\afCase{1} $R_\ep$ is nilpotent, namely $P_1(\partial,\lambda)=0$.
	
	The equality \eqref{A.BXX.B.0} always holds and \eqref{A.AXX.B.0} turns into
	\begin{align}\label{A.AXX.B.0.1}
		2\phi(\lambda)\psi_2(\partial)=\psi_1(\partial+\lambda)Q_1(\partial,\lambda).
	\end{align}
	By \eqref{A.AXX.A.0} and \eqref{A.AXX.B.0.1}, we can get $\psi_1(\partial)=0$ since $Q_1(\partial,\lambda)\neq 0$ under our assumption. Meanwhile, we also have $\phi(\lambda)=0$ or $\psi_2(\partial)=0$.
	
	\afCase{2} $R_\ep$ is solvable but not nilpotent, namely $Q_1(\partial,\lambda)=0$.
	
	By \eqref{A.BXX.B.0}, we have $\psi_1(\partial)=0$. The equality \eqref{A.AXX.B.0} becomes $2\phi(\lambda)\psi_2(\partial)=\psi_2(\partial+\lambda)P_1(\partial,\lambda)$. If $\psi_2(\partial)=0$, this equality always holds and we have $P_1(\partial,\lambda)=p(\lambda)\in\bC[\lambda]$ and $p(\lambda)\neq 0$ by Proposition~\ref{rank2.lie.con.alg}(1). Otherwise, by Lemma~\ref{lem:f(x+y)g(x,y)=f(x)h(y)}, we obtain $P_1(\partial,\lambda)=2\phi(\lambda)$ and $\psi_2(\partial)$ is a nonzero constant.
	
	So we can get the following lemma by the above discussion.
	\begin{lemma}\label{type.A.0}
		Let $R=R_\ep\oplus R_\op$ be a Lie conformal superalgebra of rank $(2+1)$ and $R_\ep$ be non-commutative. If $R_\ep$ is nilpotent, then $R$ is isomorphic to one of the following Lie conformal superalgebras\textup{:}
		
		\subno{1} $\mathcal{A}_1:=\cp A\oplus\cp B\oplus\cp X$,
		\begin{align*}
			[A_\lambda A]=f_1(\partial,\lambda)B,\ [A_\lambda X]=\phi_1(\lambda)X,\ [A_\lambda B]=[B_\lambda B]=[B_\lambda X]=[X_\lambda X]=0,
		\end{align*}
		where $f_1(\partial,\lambda)\in\bC[\partial,\lambda]$ is a nonzero skew-symmetric polynomial, $\phi_1(\lambda)\in\bC[\lambda]$.
		
		\subno{2} $\mathcal{A}_2:=\cp A\oplus\cp B\oplus\cp X$,
		\begin{align*}
			[A_\lambda A]=f_2(\partial,\lambda)B,\ [X_\lambda X]=\psi(\partial)B,\ [A_\lambda B]=[B_\lambda B]=[A_\lambda X]=[B_\lambda X]=0,
		\end{align*}
		where $f_2(\partial,\lambda)\in\bC[\partial,\lambda]$ is a nonzero skew-symmetric polynomial, $\psi(\partial)\in\cp$.
		
		If $R_\ep$ is solvable but not nilpotent, then $R$ is isomorphic to one of the following Lie conformal superalgebras\textup{:}
		
		\subno{3} $\mathcal{A}_3:=\cp A\oplus\cp B\oplus\cp X$,
		\begin{align*}
			[A_\lambda B]=2\phi_3(\lambda)B,\ [A_\lambda X]=\phi_3(\lambda)X,\ [X_\lambda X]=\eta B,\ [A_\lambda A]=[B_\lambda B]=[B_\lambda X]=0,
		\end{align*}
		where $\eta\in\bCx$, $0\neq \phi_3(\lambda)\in\bC[\lambda]$.
		
		\subno{4} $\mathcal{A}_4:=\cp A\oplus\cp B\oplus\cp X$,
		\begin{align*}
			[A_\lambda B]=p(\lambda)B,\ [A_\lambda X]=\phi_4(\lambda)X,\ [A_\lambda A]=[B_\lambda B]=[B_\lambda X]=[X_\lambda X]=0,
		\end{align*}
		where $p(\lambda), \phi_4(\lambda)\in\bC[\lambda]$ and $p(\lambda)\neq0$.
	\end{lemma}
	\begin{remark}
		For $\mathcal{A}_3$, up to isomorphism, we can take $\eta=2$ customarily. When encountering similar situations in the following, we always take this nonzero constant as 2.
	\end{remark}
	
	Our conclusion shows that a Lie conformal superalgebra may be not nilpotent even though its even part is nilpotent. However, one can easily check the following proposition.
	\begin{proposition}\label{prop:sol}
		Let $R=R_\ep\oplus R_\op$ be a Lie conformal superalgebra of rank $(2+1)$. Then $R$ is solvable if and only if $R_\ep$ is solvable. 
	\end{proposition}

	\subsection{Type $\mathcal{B}$}
	\hspace{1.5em}Let $R_\ep$ be the direct sum of two Virasoro Lie conformal algebras. We suppose that the basis $\{A,B,X\}$ satisfies the relations \eqref{H.B} and \eqref{R.rel.origin}. If $[A_\lambda X]=[B_\lambda X]=0$, we can deduce that $[X_\lambda X]=0$ by the Jacobi identity, which leads to the trivial case (see Remark~\ref{trivial case}). Otherwise, by Proposition~\ref{rank2.lie.con.alg.mod}, we can assume that $[A_\lambda X]=(\partial+\alpha\lambda+\beta)X$, $[B_\lambda X]=0$ via \eqref{H.B.mod}, since $X$ can be regarded as the module over $R_\ep$, where $\alpha,\beta\in\bC$. Using the Jacobi identity $[B_\lambda[X_\mu X]]=[[B_\lambda X]_{\lambda+\mu}X]+[X_\mu[B_\lambda X]]$ and comparing the coefficients of $B$, we obtain $ (\partial+2\lambda)\psi_2(\partial+\lambda)=0$. It follows that $\psi_2(\partial)=0$. Correspondingly, considering the Jacobi identity $[A_\lambda[X_\mu X]]=[[A_\lambda X]_{\lambda+\mu}X]+[X_\mu[A_\lambda X]]$ and comparing the coefficients of $A$, we have
	\begin{align}\label{B.AXX.A}
		(\partial+2\lambda)\psi_1(\partial+\lambda)=(\partial+2\alpha\lambda-\lambda+2\beta)\psi_1(\partial).
	\end{align}
	If $\psi_1(\partial)=0$, then the equality \eqref{B.AXX.A} always holds. Otherwise, from Lemma~\ref{lem:x+ay+b)f(x+y)=(x+2a'y-y+2b')f(x)}, we have $\alpha=\frac{3}{2}$, $\beta=0$ and $\psi_1(\partial)$ is a nonzero constant. Hence, we get the following Lie conformal superalgebras.
	\begin{lemma}
		For a Lie conformal superalgebra $R=R_\ep\oplus R_\op$ of rank $(2+1)$, if $R_\ep$ is the direct sum of two Virasoro Lie conformal algebras, then $R$ is isomorphic to one of the following Lie conformal superalgebras\textup{:}
		
		\subno{1} $\mathcal{B}_0:=\cp A\oplus\cp B\oplus\cp X$,
		\begin{align*}
			[A_\lambda A]=(\partial+2\lambda)A,\ [B_\lambda B]=(\partial+2\lambda)B,\ [A_\lambda B]=[A_\lambda X]=[B_\lambda X]=[X_\lambda X]=0,
		\end{align*}
		
		\subno{2} $\mathcal{B}_1:=\cp A\oplus\cp B\oplus\cp X$,
		\begin{align*}
			[A_\lambda A]&=(\partial+2\lambda)A,\ [B_\lambda B]=(\partial+2\lambda)B,\ [A_\lambda B]=0,\\
			[A_\lambda &X]=(\partial+\alpha\lambda+\beta)X,\ [B_\lambda X]=[X_\lambda X]=0,
		\end{align*}
		where $\alpha,\beta\in\bC$.
		
		\subno{3} $\mathcal{B}_2:=\cp A\oplus\cp B\oplus\cp X$,
		\begin{align*}
			[A_\lambda A]&=(\partial+2\lambda)A,\ [B_\lambda B]=(\partial+2\lambda)B,\ [A_\lambda B]=0,\\
			[A_\lambda &X]=(\partial+\frac{3}{2}\lambda)X,\ [X_\lambda X]=2A,\ [B_\lambda X]=0.
		\end{align*}
	\end{lemma}
	
	It is not necessary to consider the case of a module satisfying the relation \eqref{H.B.mod'}, because we can define a homomorphism between the elements of the basis $\{A, B\}$ and it is actually an isomorphism. Then, the Lie conformal superalgebra constructed by \eqref{H.B.mod'} must be isomorphic to one of $\mathcal{B}_i(i=0,1,2)$.

	\subsection{Type $\mathcal{C}$}
	\hspace{1.5em}Let $R_\ep$ be a Lie conformal algebra satisfying the relation \eqref{H.C}. Firstly, we have the trivial case (see Remark~\ref{trivial case}). Otherwise, we assume the $R_\ep$-action is nontrivial. We still suppose that $[X_\lambda X]=\psi_1(\partial)A+\psi_2(\partial)B$ just like the equality in \eqref{R.rel.origin}. In the first case, by \eqref{H.C.mod.1}, we can assume that $[A_\lambda X]=(\partial+\alpha\lambda+\beta)X$ for some $\alpha, \beta\in\bC$ and $[B_\lambda X]=0$. Using the Jacobi identity $[A_\lambda[X_\mu X]]=[[A_\lambda X]_{\lambda+\mu}X]+[X_\mu[A_\lambda X]]$ and comparing the coefficients of $A$ and $B$ respectively, we obtain 
	\begin{align}
		(\partial+2\alpha\lambda-\lambda+2\beta)\psi_1(\partial)&=(\partial+2\lambda)\psi_1(\partial+\lambda),\label{C.AXX.A}\\
		(\partial+2\alpha\lambda-\lambda+2\beta)\psi_2(\partial)&=0.\label{C.AXX.B}
	\end{align}
	We can deduce that $\psi_2(\partial)=0$ by \eqref{C.AXX.B} directly. If $\psi_1(\partial)=0$, the equality \eqref{C.AXX.A} always holds. Otherwise, by Lemma~\ref{lem:x+ay+b)f(x+y)=(x+2a'y-y+2b')f(x)}, we have $\alpha=\frac{3}{2}$, $\beta=0$ and $\psi_1(\partial)$ is a nonzero constant.
	
	In the second case, by \eqref{H.C.mod.2}, we can assume that $[B_\lambda X]=\varphi(\lambda)X$ for some nonzero polynomial $\varphi(\lambda)\in\bC[\lambda]$ and $[A_\lambda X]=0$. Using the Jacobi identity $[B_\lambda[X_\mu X]]=[[B_\lambda X]_{\lambda+\mu}X]+[X_\mu[B_\lambda X]]$, we can get $2\varphi(\lambda)\big(\psi_1(\partial)A+\psi_2(\partial)B\big)=0$, which implies that $\psi_1(\partial)=\psi_2(\partial)=0$.
	
	In a conclusion, we have the following lemma.
	\begin{lemma}
		For a Lie conformal superalgebra $R=R_\ep\oplus R_\op$ of rank $(2+1)$, if $R_\ep$ is the direct sum of the rank one commutative Lie conformal algebra and the Virasoro Lie conformal algebra, then $R$ is isomorphic to one of the following Lie conformal superalgebras\textup{:}
		
		\subno{1} $\mathcal{C}_0:=\cp A\oplus\cp B\oplus\cp X$,
		\begin{align*}
			[A_\lambda A]=(\partial+2\lambda)A,\ [A_\lambda B]=[B_\lambda B]=[A_\lambda X]=[B_\lambda X]=[X_\lambda X]=0.
		\end{align*}
		
		\subno{2} $\mathcal{C}_1:=\cp A\oplus\cp B\oplus\cp X$,
		\begin{align*}
			[A_\lambda A]=(\partial+2\lambda)A,\ [A_\lambda X]=(\partial+\alpha\lambda+\beta)X,\ [A_\lambda B]=[B_\lambda B]=[B_\lambda X]=[X_\lambda X]=0,
		\end{align*}
		where $\alpha,\beta\in\bC$.
		
		\subno{3} $\mathcal{C}_2:=\cp A\oplus\cp B\oplus\cp X$,
		\begin{align*}
			[A_\lambda A]=(\partial+2\lambda)A,\ [A_\lambda X]=(\partial+\frac{3}{2}\lambda)X,\ [X_\lambda X]=2A,\ [A_\lambda B]=[B_\lambda B]=[B_\lambda X]=0.
		\end{align*}
		
		\subno{4} $\mathcal{C}_3:=\cp A\oplus\cp B\oplus\cp X$,
		\begin{align*}
			[A_\lambda A]=(\partial+2\lambda)A,\ [B_\lambda X]=\varphi(\lambda)X,\ [A_\lambda B]=[B_\lambda B]=[A_\lambda X]=[X_\lambda X]=0,
		\end{align*}
		where $0\neq \varphi(\lambda)\in\bC[\lambda]$.
	\end{lemma}
	
	\subsection{Type $\mathcal{D}$}
	\hspace{1.5em}In this part, we discuss the case that the even part is of type D. Let $R_\ep$ be a Lie conformal algebra satisfying the relation \eqref{H.D}. Besides the trivial case, we can suppose that the $R_\ep$-action is nontrivial. By \eqref{H.D.mod}, we know that if $(a,b,Q(\partial,\lambda))\neq(1,0,0)$, then the action of $B$ is trivial, namely $[B_\lambda X]=0$. Therefore, we need to discuss the following two cases.
	
	\afCase{1} $(a,b,Q(\partial,\lambda))\neq(1,0,0)$.
	
	We can assume that $[A_\lambda X]=(\partial+\alpha\lambda+\beta)X$ for some $\alpha, \beta\in\bC$ via \eqref{H.D.mod} and $[X_\lambda X]=\psi_1(\partial)A+\psi_2(\partial)B$ just like in \eqref{R.rel.origin}. Using the Jacobi identity $[B_\lambda[X_\mu X]]=[[B_\lambda X]_{\lambda+\mu}X]+[X_\mu[B_\lambda X]]$, we have $(\partial-a\partial-a\lambda+b)\psi_1(\partial+\lambda)B=0$, namely $\psi_1(\partial)=0$. Considering the Jacobi identity $[A_\lambda[X_\mu X]]=[[A_\lambda X]_{\lambda+\mu}X]+[X_\mu[A_\lambda X]]$ and comparing the coefficients of $B$, we get
	\begin{align}\label{D.AXX.B}
		(\partial+a\lambda+b)\psi_2(\partial+\lambda)=(\partial+2\alpha\lambda-\lambda+2\beta)\psi_2(\partial).
	\end{align}
	If $\psi_2(\partial)=0$, the equality \eqref{D.AXX.B} always holds. Otherwise, by Lemma~\ref{lem:x+ay+b)f(x+y)=(x+2a'y-y+2b')f(x)}, we have $\alpha=\frac{a+1}{2}$, $\beta=\frac{b}{2}$ and $\psi_2(\partial)$ is a nonzero constant, or $a=0$, $\alpha=1$, $b=2\beta$ and $\psi_2(\partial)=c(\partial+b)$ for some $c\in\bCx$. Up to isomorphism, we can take $c=1$. Thus, we have the following lemma.
	\begin{lemma}
		Let $R=R_\ep\oplus R_\op$ be a Lie conformal superalgebra of rank $(2+1)$, and $R_\ep$ satisfies the relation \eqref{H.D}. If $(a,b,Q(\partial,\lambda))\neq(1,0,0)$, then $R$ is isomorphic to one of the following Lie conformal superalgebras\textup{:}
		
		\subno{1} $\mathcal{D}_0:=\cp A\oplus\cp B\oplus\cp X$,
		\begin{align*}
			[A_\lambda A]=(\partial+2\lambda&)A+Q_0(\partial,\lambda)B,\ [A_\lambda B]=(\partial+a_0\lambda+b_0)B,\\
			[B_\lambda B]&=[A_\lambda X]=[B_\lambda X]=[X_\lambda X]=0,
		\end{align*}
		where $a_0,b_0\in\bC$ and $Q_0(\partial,\lambda)\in\bC[\partial,\lambda]$ satisfying \textup{\textbf{C1}}.
		
		\subno{2} $\mathcal{D}_1:=\cp A\oplus\cp B\oplus\cp X$,
		\begin{align*}
			[A_\lambda A]=(\partial+2&\lambda)A+Q_1(\partial,\lambda)B,\ [A_\lambda B]=(\partial+a_1\lambda+b_1)B,\ [B_\lambda B]=0,\\ 
			[A_\lambda X]&=(\partial+\alpha\lambda+\beta)X,\ [B_\lambda X]=[X_\lambda X]=0,
		\end{align*}
		where $a_1,b_1\in\bC$ and $Q_1(\partial,\lambda)\in\bC[\partial,\lambda]$ satisfying \textup{\textbf{C1}}, $\alpha,\beta\in\bC$.
		
		\subno{3} $\mathcal{D}_2:=\cp A\oplus\cp B\oplus\cp X$,
		\begin{align*}
			[A_\lambda A]=(\partial&+2\lambda)A+Q_2(\partial,\lambda)B,\ [A_\lambda B]=(\partial+a_2\lambda+b_2)B,\ [B_\lambda B]=0,\\ 
			[A_\lambda X]&=(\partial+\frac{a_2+1}{2}\lambda+\frac{b_2}{2})X,\ [X_\lambda X]=2B,\ [B_\lambda X]=0,
		\end{align*}
		where $a_2,b_2\in\bC$ and $Q_2(\partial,\lambda)\in\bC[\partial,\lambda]$ satisfying \textup{\textbf{C1}}.
		
		\subno{4} $\mathcal{D}_3:=\cp A\oplus\cp B\oplus\cp X$,
		\begin{align*}
			[A_\lambda A]=(\partial&+2\lambda)A+Q_3(\partial,\lambda)B,\ [A_\lambda B]=(\partial+b_3)B,\ [B_\lambda B]=0,\\ 
			[A_\lambda X]&=(\partial+\lambda+\frac{b_3}{2})X,\ [X_\lambda X]=(\partial+b_3)B,\ [B_\lambda X]=0,
		\end{align*}
		where $b_3\in\bC$ and $Q_3(\partial,\lambda)\in\bC[\partial,\lambda]$ satisfying \textup{\textbf{C1}}. 
	\end{lemma}
	
	\afCase{2} $(a,b,Q(\partial,\lambda))=(1,0,0)$.
	
	Now, $R_\ep=\cp A\oplus\cp B$ is the Heisenberg--Virasoro Lie conformal algebra $\mathcal{HV}$ (see Example~\ref{HV}). In \cite{WY}, the Lie conformal superalgebra in the form of $R=R_\ep\oplus R_\op$, where $R_\ep=\mathcal{HV}$ and $R_\op=\cp X$, has been classified completely.
	\begin{lemma}\label{lem:D.2}
		\textup{(\cite{WY})} For a Lie conformal superalgebra $R=R_\ep\oplus R_\op$ of rank $(2+1)$, if $R_\ep$ is the Heisenberg--Virasoro Lie conformal algebra, then $R$ is the trivial case or isomorphic to one of the following Lie conformal superalgebras\textup{:}
		
		\subno{1} $\mathcal{D}_4:=\cp A\oplus\cp B\oplus\cp X$,
		\begin{align*}
			[A_\lambda A]=(\partial+2\lambda)A,\ [A_\lambda B]=(\partial+\lambda)B,\ [B_\lambda B]=0,\\ 
			[A_\lambda X]=(\partial+\alpha\lambda+\beta)X,\ [B_\lambda X]=\gamma X,\ [X_\lambda X]=0,
		\end{align*}
		where $\alpha,\beta,\gamma\in\bC$.
		
		\subno{2} $\mathcal{D}_5:=\cp A\oplus\cp B\oplus\cp X$,
		\begin{align*}
			[A_\lambda A]=(\partial+2\lambda)&A,\ [A_\lambda B]=(\partial+\lambda)B,\ [B_\lambda B]=0,\\ 
			[A_\lambda X]=(\partial+\lambda&)X,\ [X_\lambda X]=2B,\ [B_\lambda X]=0.
		\end{align*}
	\end{lemma}
	\begin{remark}
		We can see that $\mathcal{D}_5$ is a special case of $\mathcal{D}_2$.
	\end{remark}
	
	\subsection{Brief summary}
	\hspace{1.5em}Finally, according to Lemmas~\ref{lem:O}--\ref{lem:D.2}, we can give the main theorem in this section as follows.
	\begin{theorem}\label{thm:classification}
		Let $R=R_\ep\oplus R_\op$ be a Lie conformal superalgebra that is a free \bZz-graded \cp-module, where $R_\ep=\cp A\oplus\cp B$ and $R_\op=\cp X$. If $R$ is not the trivial case,  then it is isomorphic to one of the following Lie conformal superalgebras\textup{:}
		\[
		\begin{tabular}{cc}
			\toprule
			\phantom{1}characteristic of $R_\ep$ \phantom{12} & \phantom{12}Lie conformal superalgebra $R$\phantom{1}\\
			\midrule
			solvable & $\mathcal{A}_1,\mathcal{A}_2,\mathcal{A}_3,\mathcal{A}_4$\\ 
			via direct sum & $\mathcal{B}_1,\mathcal{B}_2,\mathcal{C}_1,\mathcal{C}_2,\mathcal{C}_3$\\ 
			others & $\mathcal{D}_1,\mathcal{D}_2,\mathcal{D}_3,\mathcal{D}_4$\\ 
			\bottomrule
		\end{tabular}
		\]
		The relations of all nontrivial cases are shown below. The parameters in different algebras are independent of each other and the omitted components vanish or are given by skew-symmetry. 
		
		\subno{1}
		$\mathcal{A}_1:\ [A_\lambda A]=f_1(\partial,\lambda)B,\ [A_\lambda X]=\phi_1(\lambda)X$, for any nonzero skew-symmetric polynomial $f_1(\partial,\lambda)\in\bC[\partial,\lambda]$, $\forall\ \phi_1(\lambda)\in\bC[\lambda]$.
		
		\subno{2}
		$\mathcal{A}_2:\ [A_\lambda A]=f_2(\partial,\lambda)B,\ [X_\lambda X]=\psi(\partial)B$, for any nonzero skew-symmetric polynomial $f_2(\partial,\lambda)\in\bC[\partial,\lambda]$, $\forall\ \psi(\partial)\in\cp$.
		
		\subno{3}
		$\mathcal{A}_3:[A_\lambda B]=2\phi_3(\lambda)B,\ [A_\lambda X]=\phi_3(\lambda)X,\ [X_\lambda X]=2B,\ \forall\ 0\neq \phi_3(\lambda)\in\bC[\lambda]$.
		
		\subno{4}
		$\mathcal{A}_4:\ [A_\lambda B]=p(\lambda)B,\ [A_\lambda X]=\phi_4(\lambda)X$,\ $\forall\ 0\neq p(\lambda)\in\bC[\lambda],\ \phi_4(\lambda)\in\bC[\lambda]$.
		
		\subno{5}
		$\mathcal{B}_1:\  [A_\lambda A]=(\partial+2\lambda)A,\ [B_\lambda B]=(\partial+2\lambda)B,\ [A_\lambda X]=(\partial+\alpha\lambda+\beta)X$, $\forall\ \alpha,\beta\in\bC$.
		
		\subno{6}
		$\mathcal{B}_2:\ [A_\lambda A]=(\partial+2\lambda)A,\ [B_\lambda B]=(\partial+2\lambda)B,\ [A_\lambda X]=(\partial+\frac{3}{2}\lambda)X,\ [X_\lambda X]=2A$.
		
		\subno{7}
		$\mathcal{C}_1:\ [A_\lambda A]=(\partial+2\lambda)A,\ [A_\lambda X]=(\partial+\alpha\lambda+\beta)X,\ \forall\ \alpha,\beta\in\bC$.
		
		\subno{8}
		$\mathcal{C}_2:\ [A_\lambda A]=(\partial+2\lambda)A,\ [A_\lambda X]=(\partial+\frac{3}{2}\lambda)X,\ [X_\lambda X]=2A$.
		
		\subno{9}
		$\mathcal{C}_3:\ [A_\lambda A]=(\partial+2\lambda)A,\ [B_\lambda X]=\varphi(\lambda)X,\ \forall\ 0\neq \varphi(\lambda)\in\bC[\lambda]$.
		
		\subno{10}
		$\mathcal{D}_1:\ [A_\lambda A]=(\partial+2\lambda)A+Q_1(\partial,\lambda)B,\ [A_\lambda B]=(\partial+a\lambda+b)B,\ [A_\lambda X]=(\partial+\alpha\lambda+\beta)X$, where $a,b\in\bC$ and $Q_1(\partial,\lambda)\in\bC[\partial,\lambda]$ satisfying \textup{\textbf{C1}}, $\forall\ \alpha,\beta\in\bC$.
		
		\subno{11}
		$\mathcal{D}_2:\ [A_\lambda A]=(\partial+2\lambda)A+Q_2(\partial,\lambda)B,\ [A_\lambda B]=(\partial+a\lambda+b)B,\ [A_\lambda X]=(\partial+\frac{a+1}{2}\lambda+\frac{b}{2})X,\ [X_\lambda X]=2B$, where $a,b\in\bC$ and $Q_2(\partial,\lambda)\in\bC[\partial,\lambda]$ satisfying \textup{\textbf{C1}}.
		
		\subno{12}
		$\mathcal{D}_3:\ [A_\lambda A]=(\partial+2\lambda)A+Q_3(\partial,\lambda)B,\ [A_\lambda B]=(\partial+b)B,\ [A_\lambda X]=(\partial+\lambda+\frac{b}{2})X,\ [X_\lambda X]=(\partial+b)B$, where $b\in\bC$ and $Q_3(\partial,\lambda)\in\bC[\partial,\lambda]$ satisfying \textup{\textbf{C1}}.
		
		\subno{13}
		$\mathcal{D}_4:\ [A_\lambda A]=(\partial+2\lambda)A,\ [A_\lambda B]=(\partial+\lambda)B,\ [A_\lambda X]=(\partial+\alpha\lambda+\beta)X,\ [B_\lambda X]=\gamma X$, $\forall\ \alpha,\beta,\gamma\in\bC$.
	\end{theorem}
	
	Up to now, we have classified all the Lie conformal superalgebras of rank $(2+1)$ completely. Based on the structures of these Lie conformal superalgebras, we can further classify them into two classes according to the $\lambda$-brackets on the odd part. 
	
	Let $R=R_\ep\oplus R_\op$ be one of the Lie conformal superalgebras $\mathcal{A}_2,\mathcal{A}_3,\mathcal{B}_2,\mathcal{C}_2,\mathcal{D}_2,\mathcal{D}_3$. Then $[{R_\op}_\lambda R_\op]\neq 0$ and the case appearing most frequently is one generator element multiplied by a fixed nonzero constant, which is closely related to the Neveu--Schwarz Lie conformal superalgebra (see Example~\ref{exa:NS}). In particular, for the Lie conformal superalgebra $\mathcal{D}_2$, if we take $(a,b,Q_2(\partial,\lambda))=(1,0,0)$, then we get the \textit{Heisenberg--Virasoro Lie conformal superalgebra}, which was introduced in \cite{CDH}. Its conformal derivations, conformal biderivations, automorphism groups and free rank $(1+1)$ conformal modules were discussed in \cite{WY}.
	
	For the rest, the odd part $R_\op$ is a solvable ideal of $R$, since $[{R_\op}_\lambda R_\op]=0$, and thus $R$ is non-semisimple. If we set the parity of $X\in R_\op$ to be even, then we get a usual Lie conformal algebra. Taking $\mathcal{D}_4$ as an example, we can get a Lie conformal algebra of rank three with free $\cp$-basis $\{A,B,X\}$ and relations
	\begin{align*}
		[A_\lambda A]=(\partial+2\lambda)A,\ [A_\lambda B]=(\partial+\lambda)B,\ [B_\lambda B]=0,\\ 
		[A_\lambda X]=(\partial+\alpha\lambda+\beta)X,\ [B_\lambda X]=\gamma X,\ [X_\lambda X]=0,
	\end{align*}
	where $\alpha,\beta\in\bC$. Actually, it is the $\mathcal{W}(a,b,r)$ Lie conformal algebra (take $\alpha=a$, $\beta=b$, $\gamma=r$) whose all finite nontrivial irreducible conformal modules were classified in \cite{LXY}. In general, any one of this class of Lie conformal superalgebras can be regarded as the \textit{super deformation} of some Lie conformal algebras. The super deformation changes the structure of the original Lie conformal algebra, for instance, it was shown that $\mathcal{D}_4$ contains the odd non-inner derivations in \cite{WY}. In terms of representation, there are many free $(1+1)$ conformal modules that can be constructed (cf. \cite{WWX,WY}). For the finite irreducible conformal modules, some interesting results will be presented later.

	\section{Automorphism groups}\label{sec:auto.groups}
	\hspace{1.5em}This section is devoted to determine the automorphism groups of Lie conformal superalgebras of rank $(2+1)$. Let $R=R_\ep\oplus R_\op$ be a Lie conformal superalgebra of rank $(2+1)$.
	
	Let $\{A,B\}$ be a basis of $R_\ep$, $\{X\}$ of $R_\op$, and take $\sigma\in\Aut(R)$. We assume that
	\begin{align*}
		\sigma(A)=f_A(\partial)A+g_A(\partial)B,\ \sigma(B)=f_B(\partial)A+g_B(\partial)B,
	\end{align*}
	where $f_A(\partial),g_A(\partial),f_B(\partial),g_B(\partial)\in\cp$. Since $\sigma(R_\ep)=R_\ep$ and $R_\ep$ is a usual rank two Lie conformal algebra, we can apply the following lemma.
	
	\begin{lemma}\label{lem:auto.1}\textup{(\cite{BCH})} Let $R_\ep$ be a rank two Lie conformal algebra defined in Proposition~\ref{rank2.lie.con.alg}. Then we have the following statements.
		
		\subno{1} If $R_\ep$ is semisimple, then we have
		\begin{align*}
			\begin{pmatrix}
				f_A(\partial) & g_A(\partial)\\
				f_B(\partial) & g_B(\partial)
			\end{pmatrix}=
			\begin{pmatrix}
				1 & 0\\
				0 & 1
			\end{pmatrix}\text{ or }
			\begin{pmatrix}
				0 & 1\\
				1 & 0
			\end{pmatrix},
		\end{align*}
		i.e., $\Aut(R_\ep)\cong \bZz$.
		
		\subno{2} If $R_\ep$ is nilpotent, then we have
		\begin{align*}
			\Aut(R_\ep)\cong\left\{\begin{pmatrix}
				k & g_A(\partial)\\
				0 & k^2
			\end{pmatrix}\ \Big|\ k\in\bCx,\ g_A(\partial)\in\cp\right\}\cong\bCx\ltimes\cp.
		\end{align*}
		
		\subno{3} If $R_\ep$ is solvable, but not nilpotent, then we have
		\begin{align*}
			\Aut(R_\ep)\cong\left\{\begin{pmatrix}
				1 & kp(-\partial)\\
				0 & k_2
			\end{pmatrix}\ \Big|\ k_2\in\bCx,\ k\in\bC\right\}\cong\bCx\ltimes\bC,
		\end{align*}
		where $p(-\partial)$ is the polynomial only depending on $R_\ep$ (see type A in Proposition~\ref{rank2.lie.con.alg}). 
		
		\subno{4} If $R_\ep$ is the Lie conformal algebra of type C, then we have
		\begin{align*}
			\Aut(R_\ep)\cong\left\{\begin{pmatrix}
				1 & 0\\
				0 & k
			\end{pmatrix}\ \Big|\ k\in\bCx\right\}\cong\bCx.
		\end{align*}
		
		\subno{5} If $R_\ep$ is the Lie conformal of type D, then we have the following conclusions\textup{:}
		
		\hspace{0.25em}\subno{i} If $b\neq 0$, then
		\begin{align*}
			\Aut(R_\ep)\cong\left\{\begin{pmatrix}
				1 & k\left(1-\frac{a-1}{b}\partial\right)\\
				0 & k_2
			\end{pmatrix}\ \Big|\ k_2\in\bCx,\ k\in\bC\right\}\cong \bCx\ltimes\bC.
		\end{align*}
		
		\hspace{0.25em}\subno{ii} If $b=0$, $a\notin\{1,0,-1\}$ and $Q(\partial,\lambda)\neq 0$, then
		\begin{align*}
			\Aut(R_\ep)\cong\left\{\begin{pmatrix}
				1 & k\partial\\
				0 & 1
			\end{pmatrix}\ \Big|\ k\in\bC\right\}\cong\bC.
		\end{align*}
		
		\hspace{0.25em}\subno{iii} If $b=0$, $a\notin\{1,0,-1\}$ and $Q(\partial,\lambda)=0$, then
		\begin{align*}
			\Aut(R_\ep)\cong\left\{\begin{pmatrix}
				1 & k\partial\\
				0 & k_2
			\end{pmatrix}\ \Big|\ k_2\in\bCx,\ k\in\bC\right\}\cong \bCx\ltimes\bC.
		\end{align*}
		
		\hspace{0.25em}\subno{iv} If $b=0$, $a\in\{1,0,-1\}$ and $Q(\partial,\lambda)\neq 0$, then
		\begin{align*}
			\Aut(R_\ep)\cong\left\{\begin{pmatrix}
				1 & g_A(\partial)\\
				0 & 1
			\end{pmatrix}\ \right\}\cong \bC^2,
		\end{align*}
		\hspace{2em}where $g_A(\partial)=a_1\partial+\delta_{a,1}a_0+\delta_{a,0}a_2\partial^2+\delta_{a,-1}a_3\partial^3$, $a_i\in\bC,\ i=0,1,2,3$.
		
		\hspace{0.25em}\subno{v} If $b=0$, $a\in\{1,0,-1\}$ and $Q(\partial,\lambda)=0$, then
		\begin{align*}
			\Aut(R_\ep)\cong\left\{\begin{pmatrix}
				1 & g_A(\partial)\\
				0 & k_2
			\end{pmatrix}\ \Big|\ k_2\in\bCx\right\}\cong \bCx\ltimes\bC^2,
		\end{align*}
		\hspace{2em}where $g_A(\partial)=a_1\partial+\delta_{a,1}a_0+\delta_{a,0}a_2\partial^2+\delta_{a,-1}a_3\partial^3$, $a_i\in\bC,\ i=0,1,2,3$.
	\end{lemma}
	
	On the odd part, we assume that $\sigma(X)=h(\partial)X$, where $h(\partial)\in\cp$. Since $\sigma(X)$ forms a basis of $R_\op$, then $h(\partial)$ is a nonzero constant. This together with Lemma~\ref{lem:auto.1}, we can get the following lemma.
	\begin{lemma}\label{lem:auto.2}
		For some $k_1,k_2,k_3\in\bCx$ and $g(\partial)\in\cp$, we have
		\begin{align*}
			\sigma\begin{pmatrix}
				A\\
				B\\
				X
			\end{pmatrix}=\begin{pmatrix}
				k_1 & g(\partial) & 0\\
				0 & k_2 & 0\\
				0 & 0 & k_3
			\end{pmatrix}
			\begin{pmatrix}
				A\\
				B\\
				X
			\end{pmatrix}.
		\end{align*}
	\end{lemma}
	
	Finally, we can give the following two main theorems of this section.
	\begin{theorem}\label{thm:auto.group}
		Let $R$ be the Lie conformal superalgebra of rank $(2+1)$. Using the notations in Theorem~\ref{thm:classification}, we have the following conclusions\textup{:}
		\begin{table}[!ht]
			\begin{minipage}{.45\textwidth}
				\centering
				\begin{tabular}{|c|c|}
					\hline
					\phantom{123456}$R$\phantom{123456}& \phantom{12345}$\Aut(R)$\phantom{12345} \\
					\hline
					$\mathcal{A}_1(\phi_1(\lambda)=0)$ & $(\bCx\ltimes\cp)\times\bCx$\\
					\hline
					$\mathcal{A}_1(\phi_1(\lambda)\neq0)$ & $\cp\times\bCx$\\
					\hline
					$\mathcal{A}_2(\psi(\partial)=0)$ & $(\bCx\ltimes\cp)\times\bCx$\\
					\hline
					$\mathcal{A}_2(\psi(\partial)\neq0)$ & $(\bCx\ltimes\cp)\times\bZz$\\
					\hline
					$\mathcal{A}_3$ & $\bCx\ltimes\bC$\\
					\hline
					$\mathcal{A}_4$ & $(\bCx\ltimes\bC)\times\bCx$\\
					\hline
				\end{tabular}
			\end{minipage}
			\begin{minipage}{.5\textwidth}
				\centering
				\begin{tabular}{|c|c|}
					\hline
					\phantom{1234}$R$\phantom{1234}& \phantom{123}$\Aut(R)$\phantom{123} \\
					\hline
					$\mathcal{B}_1$ & $\bCx$\\
					\hline
					$\mathcal{B}_2$ & $\bZz$\\
					\hline
					$\mathcal{C}_1$ & $\bCx\times\bCx$\\
					\hline
					$\mathcal{C}_2$ & $\bCx\times\bZz$\\
					\hline
					$\mathcal{C}_3$ & $\bCx$\\
					\hline  
				\end{tabular}
			\end{minipage}
		\end{table}
		
	\end{theorem}
	\begin{proof}
		Employ the notations in Lemma~\ref{lem:auto.2}. If $R_\ep$ is nilpotent, by Lemma~\ref{lem:auto.1}(2), we have $k_1=k$ and $k_2=k^2$ for some $k\in\bCx$. For $\mathcal{A}_1$, considering $\sigma[A_\lambda X]=[\sigma(A)_\lambda\sigma(X)]$ and comparing the coefficients of $X$, we can obtain $kk_3\phi_1(\lambda)=k_3\phi_1(\lambda)$. If $\phi_1(\lambda)=0$, then $k$ and $k_3$ can be any nonzero complex number. Otherwise, we have $k=1$ since $k_3\neq 0$. Thus we get the conclusion. 
		
		For $\mathcal{A}_2$, if $\psi(\partial)=0$, then $\Aut(\mathcal{A}_2)\cong\Aut(\mathcal{A}_1)$. Otherwise, using $\sigma[X_\lambda X]=[\sigma(X)_\lambda\sigma(X)]$ and comparing the coefficients of $B$, we have $k_3^2=k^2$. It follows that
		\begin{align*}
			\sigma\begin{pmatrix}
				A\\
				B\\
				X
			\end{pmatrix}=\begin{pmatrix}
				k & g(\partial) & 0\\
				0 & k^2 & 0\\
				0 & 0 & k
			\end{pmatrix}
			\begin{pmatrix}
				A\\
				B\\
				X
			\end{pmatrix}\text{ or }\begin{pmatrix}
				k & g(\partial) & 0\\
				0 & k^2 & 0\\
				0 & 0 & -k
			\end{pmatrix}
			\begin{pmatrix}
				A\\
				B\\
				X
			\end{pmatrix}
		\end{align*}
		and we get the automorphism groups of $\mathcal{A}_2$ in the table.
		
		If $R_\ep$ is solvable but not nilpotent, by Lemma~\ref{lem:auto.1}(3), we have $k_1=1$ and $g(\partial)=kp(-\partial)$ for some $k\in\bC$. Note that $\sigma[A_\lambda X]=[\sigma(A)_\lambda\sigma(X)]$ always holds for any $k_3\in\bCx$. So we can get $\Aut(\mathcal{A}_4)$ directly. For $\mathcal{A}_3$, we have $k_3^2=k_2$ by $\sigma[X_\lambda X]=[\sigma(X)_\lambda\sigma(X)]$ and $[X_\lambda X]=2B$, namely
		\begin{align*}
			\sigma\begin{pmatrix}
				A\\
				B\\
				X
			\end{pmatrix}=\begin{pmatrix}
				1 & 2k\phi(-\partial) & 0\\
				0 & k_3^2 & 0\\
				0 & 0 & k_3
			\end{pmatrix}
			\begin{pmatrix}
				A\\
				B\\
				X
			\end{pmatrix}.
		\end{align*}
		We can deduce that $\Aut(\mathcal{A}_3)\cong\bCx\ltimes\bC$.
		
		If $R_\ep$ is the direct sum of two Virasoro Lie conformal algebras, by Lemma~\ref{lem:auto.1}(1), we have 
		\begin{align*}
			\sigma\begin{pmatrix}
				A\\
				B\\
				X
			\end{pmatrix}=\begin{pmatrix}
				1 & 0 & 0\\
				0 & 1 & 0\\
				0 & 0 & k_3
			\end{pmatrix}
			\begin{pmatrix}
				A\\
				B\\
				X
			\end{pmatrix}\text{ or }\begin{pmatrix}
				0 & 1 & 0\\
				1 & 0 & 0\\
				0 & 0 & k_3
			\end{pmatrix}
			\begin{pmatrix}
				A\\
				B\\
				X
			\end{pmatrix}.
		\end{align*}
		For $\mathcal{B}_1$, we will verify the above two cases respectively. Considering $\sigma[A_\lambda X]=[\sigma(A)_\lambda\sigma(X)]$ and comparing the coefficients of $X$, we have $k_3(\partial+\alpha\lambda+\beta)=0$, which is impossible. So the second case does not exist and we get $\Aut(\mathcal{B}_1)\cong\bCx$. For $\mathcal{B}_2$, similar to $\mathcal{B}_1$, we only need to discuss the first case. Applying $\sigma$ to $[X_\lambda X]=2A$ and comparing the coefficients of $A$, we obtain $k_3^2=1$. It implies $\Aut(\mathcal{B}_2)\cong\bZz$.
		
		If $R_\ep$ is of type C, by Lemma~\ref{lem:auto.1}(4), we have 
		\begin{align*}
			\sigma\begin{pmatrix}
				A\\
				B\\
				X
			\end{pmatrix}=\begin{pmatrix}
				1 & 0 & 0\\
				0 & k & 0\\
				0 & 0 & k_3
			\end{pmatrix}
			\begin{pmatrix}
				A\\
				B\\
				X
			\end{pmatrix},
		\end{align*}
		for some $k,k_3\in\bCx$. For $\mathcal{C}_1$, we find that $k$ and $k_3$ can be taken any nonzero constants by direct computation. For $\mathcal{C}_2$, applying $\sigma$ to $[X_\lambda X]=2A$ and comparing the coefficients of $A$, we have $k_3^2=1$. For $\mathcal{C}_3$, considering $\sigma[B_\lambda X]=[\sigma(B)_\lambda\sigma(X)]$ and comparing the coefficients of $X$, we get $kk_3\varphi(\lambda)=k_3\varphi(\lambda)$. It follows that $k=1$ since $k_3\neq 0$ and $\varphi(\lambda)\neq 0$. Therefore, we get the automorphism groups of $\mathcal{C}_1$, $\mathcal{C}_2$ and $\mathcal{C}_3$, respectively.
	\end{proof}
	
	\begin{theorem}
		Using the notations in Theorem~\ref{thm:classification}, we have the following conclusions\textup{:}
		
		\subno{1} $\Aut(\mathcal{D}_1)\cong\begin{cases}
			(\bCx\ltimes\bC)\times\bCx,\ &b\neq0,\\
			\bC\times\bCx,\ &b=0,\ a\notin\{1,0,-1\},\ Q_1(\partial,\lambda)\neq0,\\
			(\bCx\ltimes\bC)\times\bCx,\ &b=0,\ a\notin\{1,0,-1\},\ Q_1(\partial,\lambda)=0,\\
			\bC^2\times\bCx,\ &b=0,\ a\in\{1,0,-1\},\ Q_1(\partial,\lambda)\neq0,\\
			(\bCx\ltimes\bC^2)\times\bCx,\ &b=0,\ a\in\{1,0,-1\},\ Q_1(\partial,\lambda)=0.
		\end{cases}$
		
		\subno{2} $\Aut(\mathcal{D}_2)\cong\begin{cases}
			\bCx\ltimes\bC,\ &b\neq0,\\
			\bC,\ &b=0,\ a\notin\{1,0,-1\},\ Q_2(\partial,\lambda)\neq0,\\
			\bCx\ltimes\bC,\ &b=0,\ a\notin\{1,0,-1\},\ Q_2(\partial,\lambda)=0,\\
			\bC^2,\ &b=0,\ a\in\{1,0,-1\},\ Q_2(\partial,\lambda)\neq0,\\
			\bCx\ltimes\bC^2,\ &b=0,\ a\in\{1,0,-1\},\ Q_2(\partial,\lambda)=0.
		\end{cases}$
		
		\subno{3} $\Aut(\mathcal{D}_3)\cong\begin{cases}
			\bCx\ltimes\bC,\ &b\neq0,\\
			\bC^2,\ &b=0,\ Q_3(\partial,\lambda)\neq0,\\
			\bCx\ltimes\bC^2,\ &b=0,\ Q_3(\partial,\lambda)=0.
		\end{cases}$
		
		\subno{4} $\Aut(\mathcal{D}_4)\cong\begin{cases}
			\bCx,\ &\gamma\neq0,\\
			(\bCx\ltimes\bC^2)\times\bCx,\ &\gamma=0.
		\end{cases}$
	\end{theorem}
	\begin{proof}
		Employ the notations in Lemma~\ref{lem:auto.2}. For (1),  we find that $\sigma[A_\lambda X]=[\sigma(A)_\lambda\sigma(X)]$ always holds for any $k_3\in\bCx$ by direct computation. Hence, the automorphism groups of $\mathcal{D}_1$ is the direct product of the automorphism groups of $R_\ep$ and $\bCx$. By Lemma~\ref{lem:auto.1}(5), we can get the conclusion.
		
		For (2), on the one hand, $\sigma[A_\lambda X]=[\sigma(A)_\lambda\sigma(X)]$ is always satisfied. On the other hand, applying $\sigma$ to $[X_\lambda X]=2B$ and comparing the coefficients of $B$, we have $k_3^2=k_2$. It implies that the automorphism groups of $\mathcal{D}_2$ is just the automorphism groups of $R_\ep$ and we get (2). Similarly, we can get (3) by putting $a=0$. 
		
		For (4), applying $\sigma$ to $[A_\lambda X]=(\partial+\alpha\lambda+\beta)X$ and $[B_\lambda X]=\gamma X$, and comparing the coefficients of $X$, respectively, by Lemma~\ref{lem:auto.1}, we obtain $k_3(-a_1\lambda+a_0)\gamma=0$ and $k_2k_3\gamma=k_3\gamma$. It follows that $\gamma=0$ or $a_0=a_1=0$ and $k_2=1$. Hence, if $\gamma=0$, then $\Aut(\mathcal{D}_4)\cong(\bCx\ltimes\bC^2)\times\bCx$. Otherwise, $\Aut(\mathcal{D}_4)\cong\bCx$.
	\end{proof}
	
	\section{Representations of Lie conformal superalgebras of rank $(2+1)$}
	\hspace{1.5em}In this section, we aim to classify the finite irreducible conformal modules over the Lie conformal superalgebras of rank $(2+1)$. Since we are only interested in finite nontrivial irreducible conformal modules (FNICMs), we always assume the conformal module is free from Lemma~\ref{lem:free.finite.rank.mod}.
	
	This section is organized roughly as follows. Firstly, we list all FNICMs over them, which is the main result in this section. Then, we prove our results and the case of type $\mathcal{D}$ is the focus of our discussion. Finally, some corollaries are presented.
	
	\subsection{Classifications of finite nontrivial irreducible conformal modules (FNICMs)}
	\hspace{1.5em}Let $R=R_\ep\oplus R_\op$ be a Lie conformal superalgebra of rank $(2+1)$, where $R_\ep=\cp A\oplus \cp B$ and $R_\op=\cp X$. Define the following conformal modules over $R$ (the parameters in different modules are independent):
	
	\subno{1} $N=\cp v$ with the action of $R$ on it satisfying
	\begin{align*}
		A_\lambda v=\mathfrak{f}(\lambda)v,\ \ B_\lambda v=\mathfrak{g}(\lambda)v,\ \ X_\lambda v=0,
	\end{align*}
	where $\mathfrak{f}(\lambda), \mathfrak{g}(\lambda)\in\bC[\lambda]$ are not zero simultaneously. If $\mathfrak{g}(\lambda)=0$ (resp., $\mathfrak{f}(\lambda)=0$), we rewrite $N$ by $N_\mathfrak{f}$ (resp., $N_\mathfrak{g}$).
	\begin{remark}
		Note that the conformal module over a Lie conformal superalgebra is also $\bZz$-graded. So, more explicitly, $N=\cp v_\ep$ or $N=\cp v_\op$. They are isomorphic to each other under the parity change. For convenience, here and from now on, we omit the parity of $v$ for any rank one conformal module.
	\end{remark}
	
	\subno{2} $M_{A,\Delta,\eta}=\cp v$ for some $\Delta\in\bCx$, $\eta\in\bC$, with $\lambda$-actions:
	\begin{align*}
		A_\lambda v=(\partial+\Delta\lambda+\eta)v,\ \ B_\lambda v=X_\lambda v=0.
	\end{align*}
	
	\subno{3} $M_{B,\Delta,\eta}=\cp v$ for some $\Delta\in\bCx$, $\eta\in\bC$, with $\lambda$-actions:
	\begin{align*}
		B_\lambda v=(\partial+\Delta\lambda+\eta)v,\ \ A_\lambda v=X_\lambda v=0.
	\end{align*}
	
	\subno{4} $M_{\Delta,\eta}^+=\cp v_\ep\oplus\cp v_\op$ for some $\Delta\in\bCx$, $\eta\in\bC$, with $\lambda$-actions:
	\begin{align*}
		\begin{cases}
			A_\lambda v_\ep=(\partial+\Delta\lambda+\eta)v_\ep,\ A_\lambda v_\op=\big(\partial+(\Delta+\frac{1}{2})\lambda+\eta\big)v_\op,\\
			B_\lambda v_\ep=B_\lambda v_\op=0,\\
			X_\lambda v_\ep=v_\op,\ X_\lambda v_\op=(\partial+2\Delta\lambda+\eta)v_\ep.
		\end{cases}
	\end{align*}
	
	\subno{5} $M_{\Delta,\eta}^-=\cp v_\ep\oplus\cp v_\op$ for some $\frac{1}{2}\neq\Delta\in\bC$, $\eta\in\bC$, with $\lambda$-actions:
	\begin{align*}
		\begin{cases}
			A_\lambda v_\ep=(\partial+\Delta\lambda+\eta)v_\ep,\ A_\lambda v_\op=\big(\partial+(\Delta-\frac{1}{2})\lambda+\eta\big)v_\op,\\
			B_\lambda v_\ep=B_\lambda v_\op=0,\\
			X_\lambda v_\ep=\big(\partial+(2\Delta-1)\lambda+\eta\big)v_\op,\ X_\lambda v_\op=v_\ep.
		\end{cases}
	\end{align*}
	
	\subno{6} $M_{\Delta,\eta,\omega}=\cp v$ for some $\Delta,\eta\in\bC,\omega\in\bCx$, with $\lambda$-actions:
	\begin{equation*}
		A_\lambda v=(\partial+\Delta\lambda+\eta)v,\ \ B_\lambda v=\omega v,\ X_\lambda v=0.
	\end{equation*}
	
	\subno{7} $M_{\Delta,\zeta,c,\epsilon}=\cp v_\ep\oplus\cp v_\op$ with $\lambda$-actions\textup{:}
	\begin{align*}
		\begin{cases}
			A_\lambda v_\ep=(\partial+\Delta\lambda+\zeta)v_\ep,\ A_\lambda v_\op=(\partial+\Delta\lambda+\zeta)v_\op,\\
			B_\lambda v_\ep=c\epsilon v_\ep,\ B_\lambda v_\op=c\epsilon v_\op,\\
			X_\lambda v_\ep=cv_\op,\phantom{c}\ X_\lambda v_\op=\epsilon v_\ep,
		\end{cases}
	\end{align*}
	where $\Delta,\zeta\in\bC$, $c,\epsilon\in\bCx$.

	For the Lie conformal superalgebra $\mathcal{D}_2$, redenote it by $\mathcal{HVS}$ if $(a,b,Q_2(\partial,\lambda))=(1,0,0)$, otherwise by $\bar{\mathcal{D}}_2$. The following theorem is our main result in this section.
	\begin{theorem}\label{thm:FNICM}
		Use the notations in Theorem~\ref{thm:classification}. Then we have the following classification of finite nontrivial irreducible conformal modules over Lie conformal superalgebra of rank $(2+1)$\textup{:}
		\begin{table}[!ht]
			\centering
			\begin{tabular}{|c|c|}    
				\hline
				Lie conformal superalgebra & \phantom{1234567}FNICM(s)\phantom{1234567}\\
				\hline
				
				$\mathcal{A}_1$, $\mathcal{A}_2$, $\mathcal{A}_3$, $\mathcal{A}_4$ & $N_\mathfrak{f}$\\ \hline
				$\mathcal{B}_1$ & $M_{A,\Delta,\eta}$, $M_{B,\Delta,\eta}$\\ \hline
				$\mathcal{B}_2$ & $M_{B,\Delta,\eta}$, $M^+_{\Delta,\eta}$, $M^-_{\Delta,\eta}$\\ \hline
				$\mathcal{C}_1$, $\mathcal{C}_3$ & $N_\mathfrak{g}$, $M_{A,\Delta,\eta}$\\ \hline
				$\mathcal{C}_2$ & $N_\mathfrak{g}$, $M^+_{\Delta,\eta}$, $M^-_{\Delta,\eta}$\\ \hline
				$\mathcal{D}_1$, $\mathcal{D}_4$ & $M_{A,\Delta,\eta}$, $M_{\Delta,\eta,\omega}$\\ \hline
				$\mathcal{HVS}$ & $M_{A,\Delta,\eta}$, $M_{\Delta,\zeta,c,\epsilon}$\\ \hline
				$\bar{\mathcal{D}}_2$, $\mathcal{D}_3$ & $M_{A,\Delta,\eta}$\\ 
				\hline
			\end{tabular}
		\end{table}
		
	\end{theorem}
	
	\subsection{FNICMs of type $\mathcal{A}$}
	\hspace{1.5em}Note that $\mathcal{A}_i(i=1,2,3,4)$ are solvable (see Proposition~\ref{prop:sol}). Let $M$ be the FNICM over $\mathcal{A}_i$. For the finite-dimensional solvable Lie algebra, we have the famous Lie's theorem. The analog of Lie's theorem in the conformal non-super case has been shown in Proposition~\ref{lem:con.Lie.thm}, and the non-conformal super case is as follows.
	
	\begin{proposition}\label{lem:lie.sup.Lie.thm}\textup{(\cite{K77})}
		Let $\mathfrak{g}=\mathfrak{g}_\ep\oplus\mathfrak{g}_\op$ be a finite-dimensional solvable Lie superalgebra and $V=V_\ep\oplus V_\op$ be its finite-dimensional irreducible module. Then either $\dim V_\ep=\dim V_\op$ and $\dim V=2^s$, where $s\le \dim\mathfrak{g}_\op$ or $\dim V=1$. 
	\end{proposition}
	
	The following lemma contains two particularly neat results of Proposition~\ref{lem:lie.sup.Lie.thm}.
	\begin{lemma}\label{lem:lie.sup.thm.app}\textup{(\cite{K77})} \subno{1} All the finite-dimensional irreducible modules of $\mathfrak{g}$ is one-dimensional if and only if $[\mathfrak{g}_\op,\mathfrak{g}_\op]\subseteq[\mathfrak{g}_\ep,\mathfrak{g}_\ep]$. 
		
		\subno{2} Let $\mathfrak{g}=\mathfrak{g}_\ep\oplus\mathfrak{g}_\op$ be a finite-dimensional nilpotent Lie superalgebra and $\mathfrak{g}^*$ its dual space. Then there is a bijective correspondence between the set of classes of isomorphic finite-dimensional irreducible $\mathfrak{g}$-modules and $\mathfrak{F}=\{ \ell\in\mathfrak{g}^*\ |\ \ell([\mathfrak{g},\mathfrak{g}])=\ell(\mathfrak{g}_\op)=0\}$.
	\end{lemma}
	
	In \cite{FKR}, the analog of Lie's theorem for finite solvable Lie conformal superalgebras was discussed. However, the relevant conclusions are difficult to be applied to the classification of the finite irreducible conformal modules. Even their ranks are not easy to be determined, which is very different from the Lie superalgebra. 
	
	Fortunately, the Lie conformal superalgebras $\mathcal{A}_i(i=1,2,3,4)$ satisfy $R_\op'\subseteq R_\ep'\subseteq\cp B$ and we have $B_\lambda v=0$ by Proposition~\ref{rank2.lie.con.alg.mod}(1), for $v\in M$. Because $0=[X_\lambda X]_{2\lambda}v=2X_\lambda(X_\lambda v)$ and $M$ is torsion-free by Lemma~\ref{lem:free.finite.rank.mod}, we can deduce that $X_\lambda v=0$ and $M$ can be considered as a conformal module over the even part of $\mathcal{A}_i$. Therefore, this together with Proposition~\ref{rank2.lie.con.alg.mod}(1), we can reach the conclusion of the solvable cases in Theorem~\ref{thm:FNICM}. 
	
	\subsection{FNICMs of type $\mathcal{B},\mathcal{C}$}
	\hspace{1.5em}This subsection is arranged to discuss the FNICMs over the Lie conformal superalgebras of type $\mathcal{B},\mathcal{C}$. Let $R^1$ be the Lie conformal superalgebra $R^1=\cp A\oplus\cp X$, where $R^1_\ep=\cp A$ and $R^1_\op=\cp X$, which satisfies
	\begin{align*}
		[A_\lambda A]=(\partial+2\lambda)A,\ \ [A_\lambda X]=(\partial+\alpha\lambda+\beta)X,\ \ [X_\lambda X]=0,
	\end{align*}
	for any $\alpha,\beta\in\bC$. It is the \textit{super deformation of Heisenberg--Virasoro type Lie conformal algebra}, which was first introduced in \cite{WWX}. Also in \cite{WWX}, its finite irreducible conformal modules were classified completely.
	\begin{proposition}\label{prop:sHV.mod}\textup{(\cite{WWX})}
		Any nontrivial conformal module over $R^1$ is of rank one and isomorphic to the following\textup{:}
		\begin{align*}
			V_{\Delta,\eta}=\cp v,\ \ A_\lambda v=(\partial+\Delta\lambda+\eta)v,\ \ X_\lambda v=0,
		\end{align*}
		for some $\Delta,\eta\in\bC$. The module $V_{\Delta,\eta}$ is irreducible if and only if $\Delta\neq 0$.
	\end{proposition}
	
	For $\mathcal{B}_1$, it is a direct sum of $R^1$ and the Virasoro Lie conformal algebra. We give the following lemma.
	\begin{lemma}\label{quasi.s.s}
		Set $R$ as a direct sum of $R^1$ and $\cp B$, that is $[A_\lambda B]=[X_\lambda B]=0$. Let $M$ be a finite conformal $R$-module. Suppose $A$ acts nontrivially on $M$. Then there exists $v\in M$ such that $A_\lambda v=(\partial+\Delta\lambda+\eta)v$ for some $\Delta, \eta\in\bC$ and $B_\lambda v=0$. 
	\end{lemma}
	\begin{proof}
		Considering that $M$ is an $R^1$-module, by Proposition~\ref{prop:sHV.mod}, there exists $v\in M$ such that $A_\lambda v=(\partial+\Delta\lambda+\eta)v$ for some $\Delta,\eta\in\bC$. Suppose that $B_\lambda v=\sum_{i\in\bZ_+}\lambda^iv_i$, where $v_i\in V$. By $[A_\lambda B]=0$, \eqref{jth prdouct} and \eqref{jth.action}, we get
		\begin{align*}
			A_{(0)}(B_\lambda v)=\sum_{i\in\bZ_+}\lambda^i(A_{(0)}v_i)=B_\lambda(A_{(0)}v)=(\partial+\lambda+\eta)\sum_{i\in\bZ_+}\lambda^iv_i.
		\end{align*}
		Comparing the coefficients of $\lambda$, we have $B_\lambda v=0$. 
	\end{proof}
	\begin{remark}\label{rem:quasi.s.s}
		Set $[A_\lambda B]=0$. By the proof of Lemma~\ref{quasi.s.s}, we can say that for any $v$ in the finite conformal module $M$, if $A_\lambda v=(\partial+\Delta\lambda+\eta)v$, then we have $B_\lambda v=0$. Conversely, if $B_\lambda v\neq 0$, then $A$ acts trivially.
	\end{remark}

	Let $M$ be the finite nontrivial irreducible conformal module over $\mathcal{B}_1$, because $0=[X_\lambda X]_{2\lambda}v=2X_\lambda(X_\lambda v)$ for any $v\in M$ and $M$ has no nonzero torsion elements by Lemma~\ref{lem:free.finite.rank.mod}, we can deduce that $X_\lambda v=0$. So $M$ can be considered as a conformal module over the even part of $\mathcal{B}_1$. By Proposition~\ref{rank2.lie.rank1.mod}, the FNICMs over $\mathcal{B}_1$ are all of rank one. Hence, the FNICM over $\mathcal{B}_1$ is of the following form and its irreducibility is given by the $\Vir$-module:
	\begin{lemma}\label{B1.mod.1}
		$M_{A,\Delta,\eta}=\cp v$ for some $\Delta\in\bCx$, $\eta\in\bC$, with $\lambda$-actions\textup{:}
		\begin{align*}
			A_\lambda v=(\partial+\Delta\lambda+\eta)v,\ \ B_\lambda v=X_\lambda v=0.
		\end{align*}
	\end{lemma}
	
	Further, besides $R^1$, $\cp B$ in $\mathcal{B}_1$ also contributes to a $\Vir$-module, where $A_\lambda v=X_\lambda v=0$ can be proved similarly. Therefore, we get another FNICM over $\mathcal{B}_1$:
	\begin{lemma}
		$M_{B,\Delta,\eta}=\cp v$ for some $\Delta\in\bCx$, $\eta\in\bC$, with $\lambda$-actions\textup{:}
		\begin{align*}
			B_\lambda v=(\partial+\Delta\lambda+\eta)v,\ \ A_\lambda v=X_\lambda v=0.
		\end{align*}
	\end{lemma}
	
	Now, we discuss the case of $\mathcal{C}_1$. It is a direct sum of $R^1$ and the rank one commutative Lie conformal algebra. Let $M$ be the finite nontrivial irreducible conformal module over $\mathcal{C}_1$. Firstly, the conformal module $M_{A,\Delta,\eta}$ defined in Lemma~\ref{B1.mod.1} is a FNICM over $\mathcal{C}_1$ by Lemma~\ref{quasi.s.s}. In addition, $\cp B$ also contributes a nontrivial conformal module. For $0\neq v\in M$, by Proposition~\ref{lem:con.Lie.thm}, we can suppose that $B_\lambda v=\mathfrak{g}(\lambda)v$ for some $0\neq\mathfrak{g}(\lambda)\in\bC[\lambda]$. $X$ acts trivially by $[X_\lambda X]=0$. By Remark~\ref{rem:quasi.s.s}, we know that $A_\lambda$ acts trivially. Hence, $M=\cp v$ by the irreducibility. We get the FNICMs over $\mathcal{C}_1$:
	\begin{lemma} 
		\subno{1} $M_{A,\Delta,\eta}=\cp v$ for some $\Delta\in\bCx$, $\eta\in\bC$, with $\lambda$-actions\textup{:}
		\begin{align*}
			A_\lambda v=(\partial+\Delta\lambda+\eta)v,\ \ B_\lambda v=X_\lambda v=0.
		\end{align*}
		
		\subno{2} $N_\mathfrak{g}=\cp v$ for some nonzero polynomial $\mathfrak{g}(\lambda)\in\bC[\lambda]$, with $\lambda$-actions\textup{:}
		\begin{align*}
			B_\lambda v=\mathfrak{g}(\lambda)v,\ \ A_\lambda v=X_\lambda v=0.
		\end{align*}
	\end{lemma}
	
	The rest of type $\mathcal{B},\mathcal{C}$ can be discussed similarly. We give the conclusions in Theorem~\ref{thm:FNICM} directly. Notice that $\mathcal{B}_2$ is semisimple. In particular, it is the direct sum of the Neveu--Schwarz Lie conformal superalgebra and the Virasoro Lie conformal algebra. Their FNICMs are clear to us. From the conclusion in \cite[Proposition 3.1]{CK}, which is a very important conclusion about the representation of the semisimple Lie conformal (super)algebra, we can get the FNICMs of $\mathcal{B}_2$ more easily.

	\subsection{FNICMs of type $\mathcal{D}$}
	\hspace{1.5em}This subsection is devoted to discuss the FNICMs over  $\mathcal{D}_i(i=1,\dots,4)$. To give the classification of FNICMs over these Lie conformal superalgebras, Cheng--Kac's lemma plays a key role.
	\begin{lemma}\label{CK.lemma}\textup{(\cite{CK})}
		Let $\mathfrak{L}$ be a Lie superalgebra over $\bC$ with a distinguished element $\partial$ and a descending sequence of subspaces $\mathfrak{L}\supset\mathfrak{L}_0\supset\mathfrak{L}_1\supset\cdots$, such that $[\partial, \mathfrak{L}_n]=\mathfrak{L}_{n-1}$ for $n\ge 1$. Let $M$ be an $\mathfrak{L}$-module and let 
		\begin{equation*}
			M_n=\{ v\in M\ |\ \mathfrak{L}_nv=0\},\ n\in\bZ_+.
		\end{equation*}
		Suppose that $M_n\neq \{0\}$ for $n\gg 0$, and that the minimal $N\in\bZ_+$ for which $M_N\neq \{0\}$ is positive. Then $\cp M_N=\cp\otimes M_N$ and therefore $\cp M_N \cap M_N=M_N$. In particular, $M_N$ is finite-dimensional if $M$ is a finitely generated $\cp$-module.
	\end{lemma}

	\subsubsection{The cases of $\mathcal{D}_1,\mathcal{D}_4$}
	\hspace{1.5em}For simplicity, consider the following Lie conformal superalgebra $\bar{\mathcal{D}}=\cp A\oplus\cp B\oplus\cp X$ with
	\begin{equation}\label{lambda.D(1)}
		\begin{split}
			&[A_\lambda A]=(\partial+2\lambda)A+Q(\partial,\lambda)B,\ [A_\lambda B]=(\partial+a\lambda+b)B,\\
			&[A_\lambda X]=(\partial+\alpha\lambda+\beta)X,\ [B_\lambda X]=\gamma X,\ [B_\lambda B]=[X_\lambda X]=0,
		\end{split}
	\end{equation}
	where $a,b\in\bC$ and $Q(\partial,\lambda)\in\bC[\partial,\lambda]$ satisfying \textup{\textbf{C1}}, $\alpha,\beta,\gamma\in\bC$. At the same time, these parameters satisfy the following conditions.
	\newline $\bullet$ If $\gamma=0$, then $\bar{\mathcal{D}}=\mathcal{D}_1$.
	\newline $\bullet$ If $(a,b,Q(\partial,\lambda))=(1,0,0)$, then $\bar{\mathcal{D}}=\mathcal{D}_4$.
	
	In the sequel, we use $d$ to denote the total degree of $Q(\partial,\lambda)$ if $Q(\partial,\lambda)$ is not a constant. Otherwise, we let $d=1$. 
	
	\begin{lemma}\label{lem:D(1)}
		The annihilation Lie superalgebra of $\bar{\mathcal{D}}$ is the Lie superalgebra $\Lie\left(\bar{\mathcal{D}}\right)^+$ which has a basis $\{A_m,\, B_n,\, X_s\ |\ m\in\bZ_{\ge-1},\ n,s\in\bZ_+\}$ over $\bC$ 
		with the following relations\textup{:}
		\begin{equation}\label{ann.D(1)}
			\begin{split}
				[A_m,A_n]&=(m-n)A_{m+n}+\sum_{i=0}^dc_iB_{m+n+2-i},\\
				[A_m,B_n]&=bB_{m+n+1}+\big((a-1)(m+1)-n\big)B_{m+n},\\
				[A_m,X_s]&=\beta X_{m+s+1}+\big((\alpha-1)(m+1)-s\big)X_{m+s},\\
				[B_n,X_s]&=\gamma X_{n+s},\ \ \ [B_m,B_n]=[X_m,X_n]=0,
			\end{split}
		\end{equation}
		where $c_i\in\bC$ for $0\le i\le d$ depending on $Q(\partial,\lambda)$. Furthermore, the extended annihilation Lie superalgebra is $\Lie\left(\bar{\mathcal{D}}\right)^e=\bC\partial\ltimes\Lie\left(\bar{\mathcal{D}}\right)^+$, which satisfies \eqref{ann.D(1)} and
		\begin{equation}\label{e.ann.D(1)}
			[\partial,A_m]=-(m+1)A_{m-1},\ [\partial,B_n]=-nB_{n-1},\ [\partial,X_s]=-sX_{s-1}.
		\end{equation}
	\end{lemma}
	\begin{proof}
		We suppose that $Q(\partial,\lambda)=\sum_{i=0}^kQ_i(\partial)\lambda^i$ for some $k\in\bZ_+$ and $Q_i(\partial)\in\cp$. By \eqref{jth prdouct} and \eqref{lambda.D(1)}, we have
		\begin{align*}
			A_{(0)}A=\partial A+Q_0(\partial)B,\ A_{(1)}A=2A+Q_1(\partial)B,\ A_{(i)}A=Q_i(\partial)i!B\ \textit{ for }\ 2\le i\le k,\\
			A_{(0)}B=(\partial+b)B,\ A_{(1)}B=aB,\ A_{(0)}X=(\partial+\beta)X,\ A_{(1)}X=\alpha X,\ B_{(0)}X=\gamma X,
		\end{align*}
		and the other items are all zero. Then, by \eqref{ann.lie.bracket}, for all $m,n\in\bZ_+$, we get
		\begin{align*}
			\quad[A_{(m)},A_{(n)}]&=\sum_{j\in\bZ_+}\binom{m}{j}(A_{(j)}A)_{(m+n-j)}\\
			&=(A_{(0)}A)_{(m+n)}+m(A_{(1)}A)_{(m+n-1)}+\sum_{j=2}^k\binom{m}{j}(A_{(j)}A)_{(m+n-j)}\\
			&=(\partial A)_{(m+n)}+m(2A)_{(m+n-1)}+\sum_{j=0}^k\binom{m}{j}\big(j!Q_j(\partial)B\big)_{(m+n-j)}\\
			&=(m-n)A_{(m+n-1)}+\sum_{i=0}^dc_iB_{(m+n-i)},
		\end{align*}
		where $c_i\in\bC$ depending on $Q(\partial,\lambda)$, $m$ and $n$. Similarly, we can deduce that
		\begin{align*}
			[A_{(m)},B_{(n)}]&=bB_{(m+n)}+\big((a-1)m-n\big)B_{(m+n-1)},\\
			[A_{(m)},X_{(s)}]&=\beta X_{(m+s)}+\big((\alpha-1)m-s\big)X_{(m+s-1)},\\
			[B_{(m)},X_{(s)}]&=\gamma X_{(m+s)},\ \ \ [B_{(n)},B_{(s)}]=[X_{(n)},X_{(s)}]=0.
		\end{align*}
		Taking $A_m=A_{(m+1)}$ for $m\ge-1$ and $B_n=B_{(n)}$, $X_n=X_{(n)}$ for $n\ge 0$, we obtain the Lie brackets in \eqref{ann.D(1)}. Finally, we get \eqref{e.ann.D(1)} by the definition of the extended annihilation Lie superalgebra in \eqref{extended.ann.lie}.
	\end{proof}
	
	Denote $\Lie\left(\bar{\mathcal{D}}\right)^e$ by $\mathcal{L}$. Next, we define $\mathcal{L}_n=\sum_{i\ge n}\left(\bC A_{i+d}\oplus\bC B_i\oplus\bC X_i\right)$ for some $n\in\bZ_+$, which is a linear subspace of $\Lie\left(\bar{\mathcal{D}}\right)^+$. Furthermore, we define $B_{-1}=X_{-1}=0$ and $\mathcal{L}_{-1}=\Lie\left(\bar{\mathcal{D}}\right)^+$. Therefore, we get a descending sequence of subspaces $\mathcal{L}\supset\mathcal{L}_{-1}\supset\mathcal{L}_{0}\supset\cdots\mathcal{L}_{n}\supset\cdots$. It is not difficult to derive the following lemma from direct computation.
	\begin{lemma}\label{ck.conditions}
		\subno{1} $[\mathcal{L}_m,\mathcal{L}_n]\subset\mathcal{L}_{m+n}$ for $m,n\in\bZ_+$. In particular, $\mathcal{L}_n$ is an ideal of $\mathcal{L}_0$ for all $n\in\bZ_+$.
		
		\subno{2} For all $n\in\bZ_+$, $[\partial,\mathcal{L}_n]=\mathcal{L}_{n-1}$.
	\end{lemma}
	
	Assume that $M$ is a FNICM over $\bar{\mathcal{D}}$. Let $M_n=\{v\in M\,|\,\mathcal{L}_n v=0\}$. By Proposition~\ref{prop:conformal}, there exists some integer $n$ such that $M_n\neq\{0\}$. If $n=-1$, then $\mathcal{L}_{-1} v=0$. We can take $0\neq v\in M_{-1}$. Then $U(\mathcal{L})v=\cp U(\mathcal{L}_{-1})v=\cp v$, where $U(\mathfrak{g})$ is the universal enveloping algebra of Lie superalgebra $\mathfrak{g}$. So, $M=\cp v$ by the irreducibility of $M$. We can deduce that $\mathcal{L}_{-1}$ acts trivially on $M$, since $\mathcal{L}_{-1}$ is an ideal of $\mathcal{L}$. Thus, $M$ is an irreducible $\cp$-module with one dimension. Equivalently, $M=\cp v$ is a trivial conformal module, which is a contradiction.
	
	Hence, we can obtain that $n\ge 0$. Note that we can take $\mathfrak{L}_m=\mathcal{L}_{m-1}$ for $m\ge 0$ here. This together with Lemma~\ref{CK.lemma} and Lemma~\ref{ck.conditions} shows that there exists a minimal integer $N\ge 1$ such that $M_N\neq\{0\}$ and $M_N$ is finite-dimensional. Now, we give the following lemma.
	\begin{lemma}\label{lem:D(1).mod.rank1}
		Let $M$ be the finite nontrivial conformal $\bar{\mathcal{D}}$-module. Then there exists a nonzero element $v\in M$ such that $\cp v$ is a conformal $\bar{\mathcal{D}}$-submodule of rank one as a $\cp$-module. In particular, if $M$ is irreducible, then $M=\cp v$.
	\end{lemma}
	\begin{proof}
		Set $\mathcal{K}=\bC(\partial-A_{-1})\oplus\mathcal{L}_0\oplus\sum_{i=0}^{d-1}\bC A_i$. Then $\mathcal{K}$ is a subalgebra of $\mathcal{L}$ and $[\mathcal{K},\mathcal{L}_n]\subset\mathcal{L}_n$ for any $n\ge 0$. Therefore $M_N$ is a nonzero nontrivial finite-dimensional $\mathcal{K}/\mathcal{L}_N$-module. Note that
		\begin{equation*}
			\mathcal{K}^{(1)}=[\mathcal{K},\mathcal{K}]\subseteq\mathcal{L}_0\oplus\sum_{i=0}^{d-1}\bC A_i,\ \mathcal{K}^{(2)}=[\mathcal{K}^{(1)},\mathcal{K}^{(1)}]\subseteq\mathcal{L}_0\oplus\sum_{i=1}^{d-1}\bC A_i,\cdots,
		\end{equation*}
		and $\mathcal{K}^{(d+1)}=[\mathcal{K}^{(d)},\mathcal{K}^{(d)}]\subseteq\mathcal{L}_0$ inductively. Moreover, 
		\begin{equation*}
			\mathcal{L}_0^{(1)}=[\mathcal{L}_0,\mathcal{L}_0]\subseteq\mathcal{L}_1\oplus\bC X_0,\ \mathcal{L}_0^{(2)}=[\mathcal{L}_0^{(1)},\mathcal{L}_0^{(1)}]\subseteq\mathcal{L}_1.
		\end{equation*}
		Then, by Lemma~\ref{ck.conditions}, we have $\mathcal{L}_0^{(n)}\subseteq\mathcal{L}_{n-1}$ for $n\ge 2$. Hence, there exists an integer $k$ depending on $d$ and $N$ such that $\mathcal{K}^{(k)}\subseteq\mathcal{L}_N$. So $\mathcal{K}/\mathcal{L}_N$ is finite-dimensional and solvable. By the super analog of Lie's theorem (see Lemma~\ref{lem:lie.sup.thm.app}(1), notice that $[X_m,X_n]=0$), the module $M_N$ is of dimension one. It follows that there exists a nonzero element $v\in M_N$ and a linear function $\ell$ on $\mathcal{K}$ such that $xv=\ell(x)v$ for all $x\in\mathcal{K}$. In other words, $A_{-1}v=(\partial+\eta)v$ for some $\eta\in\bC$, $A_i v=\ell(A_i)v$ and $B_i v=\ell(B_i)v$ for $i\ge 0$. $X_i v=0$ since the action of $X$ changes the parity for $i\ge 0$. Thus, $\cp v$ is of rank one as a conformal submodule over $\bar{\mathcal{D}}$. It is straightforward to check that $\cp v=\cp\otimes\bC v$ is a submodule of $M$. By the irreducibility of $M$ and Lemma~\ref{CK.lemma}, we get the conclusion. 
	\end{proof}
	
	\begin{proposition}\label{prop:Dbar.mod}
		The finite nontrivial conformal module over $\bar{\mathcal{D}}$ is of the form $M_{\Delta,\eta,\omega}=\cp v$ with
		\begin{equation}\label{M.delta.eta.omega}
			A_\lambda v=(\partial+\Delta\lambda+\eta)v,\ \ B_\lambda v=\omega v,\ \ X_\lambda v=0,
		\end{equation}
		where $\Delta,\eta,\omega\in\bC$, and if $(a,b,Q(\partial,\lambda))\neq(1,0,0)$, then $\omega=0$. Moreover, if $(a,b,Q(\partial,\lambda))\neq(1,0,0)$, then $M_{\Delta,\eta,\omega}=\cp v$ is irreducible if and only if $\Delta\neq 0$. If $(a,b,Q(\partial,\lambda))=(1,0,0)$, then $M_{\Delta,\eta,\omega}=\cp v$ is irreducible if and only if $\Delta\neq 0$ or $\omega\neq 0$.
	\end{proposition}
	\begin{proof}
		On the one hand, $X_\lambda$ must change the parity, which gives $X_\lambda v=0$. On the other hand, $M$ is also the conformal module over the even part of $\bar{\mathcal{D}}$, so we can get the $\lambda$-actions of $A$ and $B$ in \eqref{M.delta.eta.omega} directly by Proposition~\ref{rank2.lie.con.alg.mod}(4). In other words, $M_{\Delta,\eta,\omega}$ exhausts all the nontrivial rank one free conformal modules over $\bar{\mathcal{D}}$. Hence, we only need to show its irreducibility.
		
		If $(a,b,Q(\partial,\lambda))\neq(1,0,0)$, then $\omega=0$ by Proposition~\ref{rank2.lie.con.alg.mod}(4). $M_{\Delta,\eta,0}$ as a $\bar{\mathcal{D}}$-module is actually equivalent to $V_{\Delta,\eta}$ as a $\Vir$-module. Its irreducibility can be obtained by $\Vir$-module in Proposition~\ref{vir.mod} directly.
		
		Now let $(a,b,Q(\partial,\lambda))=(1,0,0)$. Suppose $\Delta=\omega=0$, then $M_{0,\eta,0}$ as a $\bar{\mathcal{D}}$-module is equivalent to $V_{0,\eta}$ as a $\Vir$-module. Then, by Proposition~\ref{vir.mod}, we have $(\partial+\eta)M_{0,\eta,0}\cong M_{1,\eta,0}$ and $(\partial+\eta)M_{0,\eta,0}$ is the nontrivial submodule of $M_{0,\eta,0}$, which contradicts with the irreducibility of $M_{0,\eta,0}$. Conversely, if $\Delta=0$ and $\omega\neq 0$, assume that $U$ is a nonzero submodule of $M_{0,\eta,\omega}$. Then, there exists $u=\varphi(\partial)v\in U$ for some nonzero $\varphi(\partial)\in\cp$. If $\deg\varphi(\partial)=0$, then $v\in U$ and $U=M_{0,\eta,\omega}$. If $\deg\varphi(\partial)=k>0$, then $B_\lambda u=\varphi(\partial+\lambda)(B_\lambda v)=\omega\varphi(\partial+\lambda)v\in\bC[\lambda]\otimes U$. The coefficient of $\lambda^k$ in $\omega\varphi(\partial+\lambda)v$ is a nonzero multiple of $v$, which implies $v\in U$ and $U=M_{0,\eta,\omega}$. Therefore $M_{0,\eta,\omega}$ is irreducible. If $\Delta\neq 0$, $M_{\Delta,\eta,\omega}$ is an irreducible $\Vir$-module. Furthermore, it is an irreducible $\bar{\mathcal{D}}$-module.
	\end{proof}

	\subsubsection{The case of $\mathcal{D}_2$, $\mathcal{D}_3$} 
	\hspace{1.5em}The following lemma can be proved similarly by referring to Lemma~\ref{ann.D(1)}.
	\begin{lemma}\label{lem:D2}
		The annihilation Lie superalgebra of $\mathcal{D}_2$ is the Lie superalgebra $\Lie(\mathcal{D}_2)^+$ which has a basis $\{A_m,\, B_n,\, X_s\ |\ m\in\bZ_{\ge-1},\, n,s\in\bZ_+\}$ over $\bC$ 
		with the following super-brackets\textup{:}
		\begin{equation}\label{ann.D2}
			\begin{split}
				[A_m,A_n]&=(m-n)A_{m+n}+\sum_{i=0}^dc_iB_{m+n+2-i},\\
				[A_m,B_n]&=bB_{m+n+1}+\big((a-1)(m+1)-n\big)B_{m+n},\\
				[A_m,X_s]&=\frac{1}{2}\left(bX_{m+s+1}+\big((a-1)(m+1)-2s\big)X_{m+s}\right),\\
				[X_r,X_s]&=2B_{r+s},\ \ \ [B_m,B_n]=[B_n,X_s]=0,
			\end{split}
		\end{equation}
		where $c_i\in\bC$ for $0\le i\le d$ depending on $Q_2(\partial,\lambda)$ and $d$ is its total degree. Furthermore, the extended annihilation Lie superalgebra is $\Lie(\mathcal{D}_2)^e=\bC\partial\ltimes\Lie(\mathcal{D}_2)^+$, which satisfies \eqref{ann.D2} and
		\begin{equation*}
			[\partial,A_m]=-(m+1)A_{m-1},\ [\partial,B_n]=-nB_{n-1},\ [\partial,X_s]=-sX_{s-1}.
		\end{equation*}
	\end{lemma}
	
	The following work is similar to the case of $\bar{\mathcal{D}}$. We only explain the notations and the proof is omitted. Denote $\Lie(\mathcal{D}_2)^e$ by $\mathcal{L}$. Let $\mathcal{L}_n=\sum_{i\ge n}\left(\bC A_{i+d}\oplus\bC B_i\oplus\bC X_i\right)$ for some $n\in\bZ_+$, which is a linear subspace of $\Lie(\mathcal{D}_2)$. Define $\mathcal{L}_{-1}=\Lie(\mathcal{D}_2)^+$ with $B_{-1}=X_{-1}=0$. So we obtain a descending sequence $\mathcal{L}\supset\mathcal{L}_{-1}\supset\mathcal{L}_{0}\supset\cdots\mathcal{L}_{n}\supset\cdots$ and it also satisfies the statements stated in Lemma~\ref{ck.conditions}. Let $M=M_\ep\oplus M_\op$ be a FNICM over $\mathcal{D}_2$ and $M_n=\{v\in M\,|\,\mathcal{L}_n v=0\}$. Then there exists a minimal integer $N\ge 1$ such that $M_N\neq\{0\}$ and $M_N$ is finite-dimensional by Lemma~\ref{CK.lemma}. 
	
	Set $\mathcal{Y}=\bC(\partial-A_{-1})\oplus\mathcal{L}_0\oplus\sum_{i=0}^{d}\bC A_i$. Then $\mathcal{Y}$ is a subalgebra of $\mathcal{L}$ and $[\mathcal{Y},\mathcal{L}_n]\subset\mathcal{L}_n$ for any $n\ge 0$. Thus $M_N$ is a nonzero nontrivial finite-dimensional $\mathcal{Y}/\mathcal{L}_N$-module. Denote $\mathcal{Y}/\mathcal{L}_N$ by $\mathcal{Y}_N$. By abuse of notation, we still denote the generators of $\mathcal{Y}_N$ by $\{\partial-A_{-1},\, A_n,\, B_n,\, X_n\}$, where $0\le n \le N$.
	
	Different from $\bar{\mathcal{D}}$, besides applying super analog of Lie's theorem, to get the FNICMs over $\mathcal{D}_2$, we need another Cheng--Kac's lemma:
	\begin{lemma}\label{ck.erratum}\textup{(\cite{CK})}
		Let $\mathfrak{g}$ be a finite-dimensional Lie superalgebra and let $\mathfrak{n}$ be a solvable ideal of $\mathfrak{g}$. Let $\mathfrak{a}$ be an even subalgebra of $\mathfrak{g}$ such that $\mathfrak{n}$ is a completely reducible $\mathrm{ad}\mathfrak{a}$-module with no trivial summand. Then $\mathfrak{n}$ acts trivially in any finite-dimensional $\mathfrak{g}$-module.
	\end{lemma} 
	
	Recall that we denote $\mathcal{D}_2$ with $(a,b,Q_2(\partial,\lambda))=(1,0,0)$ by $\mathcal{HVS}$ which is called the Heisenberg--Virasoro Lie conformal superalgebra, otherwise by $\bar{\mathcal{D}}_2$. 
	
	\begin{proposition} The finite nontrivial irreducible conformal module $M$ over $\mathcal{HVS}$ must be free of rank one or rank $(1+1)$.
	\end{proposition}
	\begin{proof}
		We firstly determine the $\mathcal{Y}_N$-module $M_N$. If $(a,b,Q_2(\partial,\lambda))=(1,0,0)$, then we have the following nontrivial relations on $\mathcal{Y}_N$:
		\begin{equation*}
			[A_m,A_n]=(m-n)A_{m+n},\ [A_m,B_n]=-nB_{m+n},\ [A_m,X_n]=-nX_{m+n},\ [X_m,X_n]=2B_{m+n},
		\end{equation*}
		where $-1\le m,n\le N$. Notice that $\partial-A_{-1}$ is a central element. Consider the following decomposition of $\mathcal{Y}_N$:
		\begin{equation*}
			\mathcal{Y}_N=\Span_\bC\{\partial-A_{-1},\, A_0,\, B_0,\, X_0\}+\hat{\mathcal{Y}}_N,
		\end{equation*}
		where $\hat{\mathcal{Y}}_N=\mathcal{Y}_N\backslash\Span_\bC\{\partial-A_{-1},\, A_0,\, B_0,\, X_0\}$. Obviously, $\hat{\mathcal{Y}}_N$ is a nilpotent ideal of $\mathcal{Y}_N$. Consider the action of $A_0$ on $\hat{\mathcal{Y}}_N$:
		\begin{equation*}
			[A_0,A_n]=-nA_{n},\ \ [A_0,B_n]=-nB_{n},\ \ [A_0,X_n]=-nX_{n},
		\end{equation*}
		where $1\le n\le N$. It implies that $\hat{\mathcal{Y}}_N$ is a completely reducible $\bC A_0$-module with no trivial summand. By Lemma~\ref{ck.erratum}, $\hat{\mathcal{Y}}_N$ acts trivially on $M_N$. In other words, $M_N$ is determined by $\Span_\bC\{\partial-A_{-1},\, A_0,\, B_0,\, X_0\}$. Notice that $\Span_\bC\{\partial-A_{-1},\, A_0,\, B_0\}$ is the even subalgebra and its module is of dimension one (its conformal analog is just Proposition~\ref{rank2.lie.con.alg.mod}(4), one may refer to \cite[Theorem~3.10]{LHW} for more details). Hence, we can take nonzero $v$ in this one-dimensional module. By $2X^2_0v=[X_0,X_0]v=2B_0v$, we immediately obtain that  $M_N=\bC v+\bC X_0 v$ and $|v|\neq |X_0 v|$. Therefore, $M=\cp M_N$ and $\rank\, M\le 2$ by the irreducibility of $M$ and Lemma~\ref{CK.lemma}. More explicitly, $M$ is of rank one or rank $(1+1)$.
	\end{proof}
	
	\begin{proposition} The finite nontrivial irreducible conformal module $M$ over $\bar{\mathcal{D}}_2$ must be free of rank one.
	\end{proposition}
	\begin{proof}
		Now we have $(a,b,Q_2(\partial,\lambda))\neq(1,0,0)$. Let $R=R_\ep\oplus R_\op$ be the image of $\mathcal{Y}_N$ in $\End M_N$, which is obviously finite-dimensional. By abuse of notation, we still use  $\{\partial-A_{-1},\, A_n,\, B_n,\, X_n\}$ to denote its image in $R$, where $0\le n \le N$. Moreover, we have the Lie superalgebra $\tilde{R}=\bC(\partial-A_{-1})\oplus\sum_{i=1}^N \bC A_i\oplus\sum_{i=0}^N(\bC B_i\oplus\bC X_i)$, which is a nilpotent ideal of $R$. By Lemma~\ref{lem:lie.sup.thm.app}(2), there exists some nonzero $v\in M_N$ and a linear function $\ell:\,\tilde{R}\rightarrow\bC$ satisfying $xv=\ell(x)v$ for any $x\in\tilde{R}$, where $X_iv=0$ for all $i\ge 0$. We denote this one-dimensional module of the subalgebra $\tilde{R}$ by $V=\bC v$. In other words, the odd part of $R$ acts trivially on $V$. Then $M_N=\bC A_0 V$ is also of dimension one by \cite[Theorem~3.10]{LHW}.
		
		Finally, $M=\cp M_N$ is of rank one by the irreducibility of $M$ and Lemma~\ref{CK.lemma}.
	\end{proof}
	
	Let $M$ be the nontrivial rank one conformal module over $\mathcal{D}_2$. As the nontrivial conformal module over the even part, it is of the form \eqref{H.D.mod}. For any $v\in M$, we have $X_\lambda v=0$ since $X_\lambda$ change the parity. Besides, we can get $B_\lambda v=0$ from $0=[X_\lambda X]_{\lambda+\mu}v=2B_{\lambda+\mu}v$. So the $\lambda$-actions of $M$ are given as follows:
	\begin{equation}\label{D2.mod.rank1}
		A_\lambda v=(\partial+\Delta\lambda+\eta)v,\ \ B_\lambda v=X_\lambda v=0,
	\end{equation}
	for some $\Delta,\eta\in\bC$. It is irreducible if and only if $\Delta\neq 0$, which is given by the irreducibility of $\Vir$-module.
	
	Therefore, in the following, we need to discuss the nontrivial conformal modules of rank $(1+1)$ over $\mathcal{HVS}$. Let $(a,b,Q_2(\partial,\lambda))=(1,0,0)$, then the relations of $\mathcal{D}_2$ turn into
	\begin{align*}
		[A_\lambda A]=(\partial+2\lambda)A,\ [A_\lambda B]=&(\partial+\lambda)B,\ [A_\lambda X]=(\partial+\lambda)X,\\
		[X_\lambda X]=2B,\ [B_\lambda &B]=[B_\lambda X]=0.
	\end{align*}
	Its free conformal modules of rank $(1+1)$ were constructed in \cite{WY}.
	\begin{proposition}\label{HVS.mod}\textup{(\cite{WY})}
		Up to parity change, any nontrivial free conformal module of rank $(1+1)$ over the Heisenberg--Virasoro Lie conformal superalgebra $\mathcal{HVS}$ is isomorphic to one of the followings\textup{:}
		
		\subno{1} $M_{\Delta,\zeta,c,\epsilon}=\cp v_\ep\oplus\cp v_\op$ with $\lambda$-actions\textup{:}
		\begin{align*}
			\begin{cases}
				A_\lambda v_\ep=(\partial+\Delta\lambda+\zeta)v_\ep,\ A_\lambda v_\op=(\partial+\Delta\lambda+\zeta)v_\op,\\
				B_\lambda v_\ep=c\epsilon v_\ep,\ B_\lambda v_\op=c\epsilon v_\op,\\
				X_\lambda v_\ep=cv_\op,\phantom{c}\ X_\lambda v_\op=\epsilon v_\ep,
			\end{cases}
		\end{align*}
		where $\Delta,\zeta\in\bC$, $c,\epsilon\in\bCx$.
		
		\subno{2} $M_{\Delta_0,\Delta_1,\eta}^{(i)}=\cp v_\ep\oplus\cp v_\op\ (i=1,2,\dots,6)$ with $\lambda$-actions\textup{:}
		\begin{align*}
			\begin{cases}
				A_\lambda v_\ep=(\partial+\Delta_0\lambda+\eta)v_\ep,\ A_\lambda v_\op=(\partial+\Delta_1\lambda+\eta)v_\op,\\
				B_\lambda v_\ep=B_\lambda v_\op=0,\\
				X_\lambda v_\ep=h(\partial,\lambda)v_\op,\ X_\lambda v_\op=0,
			\end{cases}
		\end{align*}
		satisfying the following relations\textup{:}
		\[
		\begin{tabular}{cccc}
			\toprule
			\phantom{12}\text{conformal module}\phantom{12} & \phantom{1234}$\Delta_0$\phantom{1234} & \phantom{1234}$\Delta_1$\phantom{1234} & \phantom{12}$h(\partial,\lambda)$\phantom{12}\\
			\midrule
			$M_{\Delta_0,\Delta_1,\eta}^{(1)}$ & $\Delta_0$ & $\Delta_0$ & $1$\\
			$M_{\Delta_0,\Delta_1,\eta}^{(2)}$ & $\Delta_0$ & $\Delta_0-1$ & $\lambda$\\
			$M_{\Delta_0,\Delta_1,\eta}^{(3)}$ & $\Delta_0$ & $\Delta_0-2$ & $\lambda(\partial-\Delta_1\lambda+\eta)$\\
			$M_{\Delta_0,\Delta_1,\eta}^{(4)}$ & $1$ & $0$ & $\partial+k\lambda+\eta,k\in\bC$\\
			$M_{\Delta_0,\Delta_1,\eta}^{(5)}$ & $1$ & $-2$ & $\lambda(\partial+\lambda+\eta)(\partial+2\lambda+\eta)$\\
			$M_{\Delta_0,\Delta_1,\eta}^{(6)}$ & $3$ & $0$ & $\lambda(\partial+\eta)(\partial-\lambda+\eta)$\\
			\bottomrule
		\end{tabular}
		\]
		where $\Delta_0,\Delta_1,\eta\in\bC$ and $h(\partial,\lambda)\in\bC[\partial,\lambda]$.
	\end{proposition}
	
	\begin{remark}
		\subno{1} For $M_{\Delta_0,\Delta_1,\eta}^{(i)}$ defined in Proposition~\ref{HVS.mod}, it is easy to check that $\cp v_\op$ is its submodule of rank one. Then the quotient $M/\cp v_\op$ is of the form in \eqref{D2.mod.rank1}. 
		
		\subno{2} The irreducibility of $M_{\Delta,\zeta,c,\epsilon}$ can be given by the irreducibility of Heisenberg--Virasoro conformal module. Employing the notations in Example~\ref{HV}, then the finite nontrivial irreducible conformal $\mathcal{HV}$-module is of the following form (cf. \cite{WY17}):
		\begin{equation*}
			V_{\Delta,\eta,\omega}=\cp v,\ \ L_\lambda v=(\partial+\Delta\lambda+\eta)v,\ \ H_\lambda v=\omega v,
		\end{equation*}
		where $\Delta,\eta,\omega\in\bC$ and $(\Delta,\omega)\neq (0,0)$. For $M_{\Delta,\zeta,c,\epsilon}$, its restriction to its even subalgebra is the $\mathcal{HV}$-module. Notice that $c\epsilon\neq 0$, so we can deduce that it is irreducible. 
	\end{remark}
	
	By the above propositions and remark, we get the FNICMs over the Lie conformal superalgebra $\mathcal{D}_2$ listed in Theorem~\ref{thm:FNICM}. The case of $\mathcal{D}_3$ can be discussed similarly. We only show the conclusions.
	\begin{proposition}
		The finite nontrivial irreducible conformal module $M$ over $\mathcal{D}_3$ must be free of rank one. More explicitly, 
		\begin{align*}
			M=\cp v,\ \ A_\lambda v=(\partial+\Delta\lambda+\eta)v,\ \ B_\lambda v=X_\lambda v=0,
		\end{align*}
		where $\Delta,\eta\in\bC$, and it is irreducible if and only if $\Delta\neq 0$.
	\end{proposition}
	
	Therefore, we complete the proof of Theorem~\ref{thm:FNICM}.
	
	\subsection{Corollaries}
	\hspace{1.5em}At the end of this paper, we discuss a special class of Lie conformal superalgebras. Let $R$ be the Lie conformal algebra of finite rank and contain a generator $X$ satisfying $[X_\lambda X]=0$. Set $X$ to be odd and we get a new Lie conformal superalgebra $R^\mathfrak{s}=R^\mathfrak{s}_\ep\oplus R^\mathfrak{s}_\op$, where $R^\mathfrak{s}_\ep=R\backslash\cp X$ and $R^\mathfrak{s}_\op=\cp X$. One can easily verify that it satisfies the axioms in Definition~\ref{def:lie conformal superalgebra}. Perhaps we can call it the \textit{single super deformation} of a Lie conformal algebra. Let $M$ be the FNICM of $R^\mathfrak{s}$. Note that $0=[X_\lambda X]_{2\lambda}v=2X_\lambda(X_\lambda v)$ for any $v\in M$. By Lemma~\ref{lem:free.finite.rank.mod}, $M$ is torsion-free. So we can deduce that $X_\lambda v=0$ and get the following corollary.
	\begin{corollary}
		Let $R$ be a Lie conformal algebra. Suppose that $R^\mathfrak{s}$ exits and $M$ is the FNICM over $R^\mathfrak{s}$. Then $M$ is a FNICM over $R^\mathfrak{s}_\ep$ with $X_\lambda v=0$.
	\end{corollary}
	
	Recall an important theorem in \cite{XHW}, which shows that any FNICM of a finite Lie conformal algebra $R$ satisfying $R/\Rad(R)=\Vir$ is free of rank one, where $\Rad(R)$ is the radical of $R$.
	\begin{theorem}\label{thm:rk1.mod}\textup{(\cite{XHW})}
		Let $R$ be a free finite Lie conformal algebra such that $R/\Rad(R)=\Vir$. Then any finite nontrivial conformal $R$-module $M$ must contain a free rank one conformal $R$-submodule.
	\end{theorem}
	Thus, we have the following corollary:
	\begin{corollary}
		Let $R$ be the Lie conformal algebra defined in Theorem~\ref{thm:rk1.mod}. Then any FNICM over $R^\mathfrak{s}$ if $R^\mathfrak{s}$ exists is free of rank one.
	\end{corollary}
	
	For example, according to Theorem~\ref{thm:classification}, we know that the Lie conformal superalgebra of $\mathcal{B}_1,\mathcal{C}_1,\mathcal{C}_3,\mathcal{D}_1,\mathcal{D}_4$ is the single super deformation of some Lie conformal algebra. By Theorem~\ref{thm:FNICM}, the FNICMs over them are all of rank one and the $\lambda$-actions of $X$ are trivial.

\end{document}